\documentclass[a4paper,12pt]{amsart}
\title[Stability conditions and curve counting invariants]{{\bf Stability conditions and curve counting invariants
on Calabi-Yau 3-folds}}
\date{}
\author{Yukinobu Toda}

\usepackage{amscd}
\usepackage{amsmath}
\usepackage{amssymb}
\usepackage{amsthm}
\usepackage{float}
\usepackage[dvips]{graphicx}

\usepackage[all,ps,dvips]{xy}

\usepackage{array}
\usepackage{amscd}
\usepackage[all]{xy}

\DeclareFontFamily{U}{rsfs}{%
\skewchar\font127}
\DeclareFontShape{U}{rsfs}{m}{n}{%
<-6>rsfs5<6-8.5>rsfs7<8.5->rsfs10}{}
\DeclareSymbolFont{rsfs}{U}{rsfs}{m}{n}
\DeclareSymbolFontAlphabet
{\mathrsfs}{rsfs}
\DeclareRobustCommand*\rsfs{%
\@fontswitch\relax\mathrsfs}

\theoremstyle{plain}
\newtheorem{thm}{Theorem}[section]
\newtheorem{prop}[thm]{Proposition}
\newtheorem{lem}[thm]{Lemma}

\newtheorem{defi}[thm]{Definition}
\newtheorem{rmk}[thm]{Remark}
\newtheorem{cor}[thm]{Corollary}

\newtheorem{prop-defi}[thm]{Proposition-Definition}
\newtheorem{thm-defi}[thm]{Theorem-Definition}
\newtheorem{lem-defi}[thm]{Lemma-Definition}

\newtheorem{conj}[thm]{Conjecture}
\newtheorem{exam}[thm]{Example}

\newdimen\argwidth
\def\db[#1\db]{
 \setbox0=\hbox{$#1$}\argwidth=\wd0
 \setbox0=\hbox{$\left[\box0\right]$}
  \advance\argwidth by -\wd0
 \left[\kern.3\argwidth\box0 \kern.3\argwidth\right]}

\newcommand{\aA}{\mathcal{A}}

\newcommand{\cC}{\mathcal{C}}
\newcommand{\dD}{\mathcal{D}}
\newcommand{\eE}{\mathcal{E}}

\newcommand{\hH}{\mathcal{H}}

\newcommand{\lL}{\mathcal{L}}
\newcommand{\mM}{\mathcal{M}}

\newcommand{\oO}{\mathcal{O}}

\newcommand{\sS}{\mathcal{S}}

\newcommand{\uU}{\mathcal{U}}
\newcommand{\vV}{\mathcal{V}}

\newcommand{\yY}{\mathcal{Y}}

\newcommand{\Supp}{\mathop{\rm Supp}\nolimits}
\newcommand{\Hom}{\mathop{\rm Hom}\nolimits}

\newcommand{\dR}{\mathbf{R}}
\newcommand{\dL}{\mathbf{L}}

\newcommand{\Hilb}{\mathop{\rm Hilb}\nolimits}

\newcommand{\ch}{\mathop{\rm ch}\nolimits}

\newcommand{\Ext}{\mathop{\rm Ext}\nolimits}
\newcommand{\Spec}{\mathop{\rm Spec}\nolimits}
\newcommand{\rank}{\mathop{\rm rank}\nolimits}
\newcommand{\Coh}{\mathop{\rm Coh}\nolimits}

\newcommand{\cneq}{\mathrel{\raise.095ex\hbox{:}\mkern-4.2mu=}}
\newcommand{\eqcn}{\mathrel{=\mkern-4.5mu\raise.095ex\hbox{:}}}

\newcommand{\Cok}{\mathop{\rm Cok}\nolimits}

\newcommand{\Aut}{\mathop{\rm Aut}\nolimits}

\newcommand{\Stab}{\mathop{\rm Stab}\nolimits}

\newcommand{\DT}{\mathop{\rm DT}\nolimits}
\newcommand{\PT}{\mathop{\rm PT}\nolimits}

\newcommand{\Imm}{\mathop{\rm Im}\nolimits}

\newcommand{\Ker}{\mathop{\rm ker}\nolimits}

\newcommand{\GL}{\mathop{\rm GL}\nolimits}

\newcommand{\tr}{\mathop{\rm tr}\nolimits}
\newcommand{\ex}{\mathop{\rm ex}\nolimits}

\newcommand{\cl}{\mathop{\rm cl}\nolimits}


\begin{document}
\maketitle

\begin{abstract}
The purpose of this paper
is twofold:
first we give a survey on 
the recent developments of 
curve counting invariants 
on Calabi-Yau 3-folds, e.g.
Gromov-Witten theory, Donaldson-Thomas
theory and Pandharipande-Thomas theory.  
Next we focus on the proof of the rationality 
conjecture of 
the generating series of
PT
invariants, and discuss its conjectural
Gopakumar-Vafa form.  
\end{abstract}
\section{Introduction}
\subsection{Background}
Let $X$ be a smooth projective
Calabi-Yau 3-fold, i.e. 
\begin{align*}
\bigwedge^3 T_{X}^{\vee} \cong \oO_X, \ 
H^1(X, \oO_X)=0. 
\end{align*}
We are interested in the curve counting 
theory on $X$. This is an 
important field of study in
connection with mirror symmetry: 
it predicts
a relationship between curve counting 
invariants on $X$ and a period integral 
on its mirror manifold $\check{X}$. 
So far curve counting invariants have 
been computed and 
compared under the mirror symmetry
in several situations. 

Now there are three kinds of
 curve counting theories on $X$:
\begin{itemize}
\item {\bf Gromov-Witten (GW) theory: }
counting pairs, 
\begin{align*}
(C, f), \ f \colon C \to X, 
\end{align*}
where $C$ is a connected nodal curve and $f$ 
is a morphism with finite automorphisms. 
In terms of string theory, 
GW invariants count \textit{world sheets}.
The moduli space defining the GW theory 
is Kontsevich's stable map moduli space. 
The resulting invariants are $\mathbb{Q}$-valued. 
\item {\bf Donaldson-Thomas (DT) theory: }
counting subschemes, 
\begin{align*}
Z \subset X, 
\end{align*} 
with $\dim Z \le 1$. 
In terms of string theory, DT
invariants count \textit{D-branes}. 
The moduli space defining the
DT theory is the classical Hilbert scheme.
The resulting invariants are $\mathbb{Z}$-valued. 
\item {\bf Pandharipande-Thomas (PT) theory: }
counting pairs, 
\begin{align*}
(F, s), \quad s \colon \oO_X \to F, 
\end{align*}
where $F$ is a pure one dimensional sheaf, and $s$
is surjective in dimension one. 
The PT invariants also count D-branes, but 
the stability condition is different 
from DT theory. 
The moduli space defining the PT theory is 
identified with the moduli space of two term complexes, 
\begin{align*}
I^{\bullet}=(\oO_X \stackrel{s}{\to} F) \in D^b \Coh(X). 
\end{align*}
Here $D^b \Coh(X)$ is the bounded derived category of 
coherent sheaves on $X$.
\end{itemize}
An equivalence between 
GW and DT theories is conjectured by 
Maulik-Nekrasov-Okounkov-Pandharipande~\cite{MNOP}.
Also an equivalence between 
DT and PT theories 
is conjectured by Pandharipande-Thomas~\cite{PT}.
They are formulated in terms of generating functions. 

On the other hand, the notion of stability conditions 
on $D^b \Coh(X)$ is introduced by Bridgeland~\cite{Brs1}. 
He shows that the set of stability conditions on 
$D^b \Coh(X)$, denoted by 
\begin{align*}
\Stab(X),
\end{align*}
has a structure of
 a complex manifold. The space $\Stab(X)$
is expected to be related to the stringy 
K$\ddot{\rm{a}}$hler moduli space, which 
should be isomorphic to the moduli space of 
complex structures of the mirror $\check{X}$. 
An important observation by Pandharipande-Thomas~\cite{PT} 
is that the DT/PT correspondence should be 
interpreted as
wall-crossing phenomena 
in the
space of stability conditions
$\Stab(X)$. 
Although it is still difficult to study 
$\Stab(X)$ when $X$ is a projective Calabi-Yau 3-fold, 
kinds of `limiting degenerations' of 
Bridgeland stability have been 
introduced in~\cite{Bay}, \cite{Tolim}, \cite{Tcurve1}, 
and DT/PT wall-crossing is also 
observed in these degenerated stability conditions. 

In recent years,  
the wall-crossing formula
of DT type invariants have been established 
by Joyce-Song~\cite{JS}
and Kontsevich-Soibelman~\cite{K-S}
in a general setting. 
Since then, it turns out that 
a categorical approach is useful 
in the study of 
DT type 
curve counting invariants.  
Now several applications have been
obtained, e.g. DT/PT correspondence,
rationality conjecture. (cf.~\cite{BrH},~\cite{StTh}, 
~\cite{Tcurve1}, \cite{Tolim2}.)
One of the purposes of this paper is to give a survey 
of these recent developments. 

As for another purpose, we 
focus on the 
rationality 
conjecture of the generating series of 
PT invariants proposed in~\cite{PT}. 
The Euler characteristic version is proved in~\cite{Tolim2}, 
and the virtual cycle is involved in~\cite{BrH}.  
In this paper, 
assuming the announced
result by Behrend-Getzler~\cite{BG},
we give its another proof 
by discussing in the framework of~\cite{Tcurve1}.
The main idea is the same as in~\cite{Tolim2}, 
but the argument is simplified. 
We also discuss a conjectural 
Gopakumar-Vafa form of the generating series of 
PT invariants, and see that it is 
related to the multi-covering formula of 
generalized DT invariants introduced by 
Joyce-Song~\cite{JS}. 
We also give an 
 evidence of the conjectural multi-covering formula
when $X$ is a certain elliptically fibered Calabi-Yau 3-fold. 

\subsection{Plan of the paper}
In Section~\ref{sec:Stab}, 
we give a survey on stability 
conditions. 
In Section~\ref{sec:Curve}, we recall 
several curve counting invariants on 
Calabi-Yau 3-folds and the relevant 
conjectures, results. In Section~\ref{sec:Hall},
we recall the notion of Hall algebras and 
the generalized DT invariants counting 
one dimensional sheaves. 
In Section~\ref{rational}, we give a 
proof of the rationality of the 
generating series of PT invariants in 
the framework of~\cite{Tcurve1}.
In Section~\ref{sec:Prod}, we discuss
a Gopakumar-Vafa form of the generating 
series of PT invariants, and the 
multi-covering formula of generalized 
DT invariants.  

\subsection{Acknowledgement}
The author would like to thank the referee
for checking the manuscript carefully and 
give useful comments. 
This work is supported by World Premier 
International Research Center Initiative
(WPI initiative), MEXT, Japan. This work is also supported by Grant-in Aid
for Scientific Research grant (22684002), 
and partly (S-19104002),
from the Ministry of Education, Culture,
Sports, Science and Technology, Japan.

\subsection{Notation and Convention}
For a triangulated category $\dD$, 
the shift functor is denoted by $[1]$. 
For a set of objects $\sS \subset \dD$, 
we denote by $\langle \sS \rangle_{\tr}$
the smallest triangulated subcategory 
which contains $\sS$ and $0 \in \dD$. 
Also we denote by $\langle \sS \rangle_{\ex}$
the smallest extension closed 
subcategory of $\dD$ which contains 
$\sS$ and $0 \in \dD$. 
The abelian category of coherent sheaves on 
a variety $X$ is denoted by 
$\Coh(X)$. We say 
$F \in \Coh(X)$ is 
$d$-dimensional if its support is 
$d$-dimensional. 
We always assume that the second 
homology group $H_2(X, \mathbb{Z})$ is 
torsion free. If there is a torsion, 
then the arguments are applied 
if we replace $H_2(X, \mathbb{Z})$ by 
its torsion free part. 
For $\beta \in H_2(X, \mathbb{Z})$, we 
write $\beta>0$ if $\beta$ is a
class of an effective algebraic 
one cycle on $X$.

\section{Stability conditions}\label{sec:Stab}
We begin with recalling stability conditions 
on abelian categories, and explain typical 
wall-crossing phenomena. 
\subsection{Definitions of stability conditions}
\label{subsec:definition}
Classically there is a notion of a
stability condition on vector bundles on
smooth 
projective curves. Let $C$ be a smooth 
projective curve over $\mathbb{C}$ and 
$E$ a vector bundle on it. 
The slope of $E$ is defined by 
\begin{align*}
\mu(E) \cneq \deg(E)/\rank(E). 
\end{align*}
\begin{defi}\label{sbund}
A vector bundle $E$ on $C$ is (semi)stable 
if for any subbundle $0\neq F \subsetneq E$, we have 
\begin{align*}
\mu(F)<(\le) \mu(E). 
\end{align*}
\end{defi}
We have the following properties:
\begin{itemize}
\item If we fix rank $r$ and degree $d$, 
then there is a good moduli space of 
slope semistable vector bundles $E$ 
with $\rank(E)=r$ and $\deg(E)=d$. 
\item For any vector bundle $E$ on $C$, there is a filtration,
(Harder-Narasimhan filtration,) 
\begin{align*}
0=E_0 \subset E_1 \subset \cdots \subset E_N=E, 
\end{align*}
such that each subquotient $F_i=E_i/E_{i-1}$
is semistable with $\mu(F_i)>\mu(F_{i+1})$
for all $i$. 
\end{itemize}
A stability condition on an abelian category 
is defined to be a direct generalization of the 
above classical notion. 
Let $\aA$ be an abelian category, e.g.
the category of coherent 
sheaves on an algebraic variety.  
Recall that its Grothendieck group is defined by 
\begin{align*}
K(\aA) \cneq \bigoplus_{E\in\aA} \mathbb{Z}[E]/\sim, 
\end{align*}
where the equivalence relation $\sim$ is generated by 
\begin{align*}
[E_2] \sim [E_1]+[E_3],
\end{align*}
for all exact sequences $0\to E_1 \to E_2 \to E_3 \to 0$
in $\aA$. 
We fix a finitely generated abelian group $\Gamma$
together with a group homomorphism, 
\begin{align*}
\cl \colon K(\aA) \to \Gamma.
\end{align*}
For instance if $\aA=\Coh(X)$ for a 
smooth projective variety $X$, we can take $\Gamma$
to be the image of the Chern character map, 
\begin{align}\label{cl=ch}
\ch \colon K(\aA) \twoheadrightarrow \Gamma \subset H^{\ast}(X, \mathbb{Q}),
\end{align}
and $\cl=\ch$. 
Let $\mathbb{H} \subset \mathbb{C}$ be the subset
\begin{align*}
\mathbb{H} = \{ r\exp(\pi i \phi) : r>0, 0<\phi \le 1\}. 
\end{align*}
The following formulation of stability conditions is 
due to Bridgeland~\cite{Brs1}. 
\begin{defi}\label{def:A}
A stability condition on $\aA$ is a group homomorphism, 
\begin{align*}
Z \colon \Gamma \to \mathbb{C}, 
\end{align*}
satisfying the following axiom. 

(i) For any non-zero object $E\in \aA$, we have 
\begin{align*}
Z(E) \cneq Z(\cl(E)) \in \mathbb{H}.
\end{align*}
In particular the argument 
\begin{align*}
\arg Z(E) \in (0, \pi],
\end{align*}
is well-defined. An object $E\in \aA$ is called
$Z$-(semi)stable if for any non-zero subobject 
$0\neq F \subsetneq E$, we have 
\begin{align*}
\arg Z(F) <(\le) \arg Z(E). 
\end{align*}

(ii) For any object $E\in \aA$, there is a filtration, 
(Harder-Narasimhan filtration,)
\begin{align*}
0=E_0 \subset E_1 \subset \cdots \subset E_N=E, 
\end{align*}
such that each subquotient $F_i=E_i/E_{i-1}$ is 
$Z$-semistable with 
\begin{align*}
\arg Z(F_1)> \arg Z(F_2)> \cdots >\arg Z(F_N). 
\end{align*}
\end{defi}
Here we give some examples. 
\begin{exam}\label{exam:stab}
(i) Let $C$ be a smooth projective curve over 
$\mathbb{C}$ and take $\aA=\Coh(C)$. 
We set $\Gamma$ to be 
\begin{align*}
\Gamma= \mathbb{Z} \oplus \mathbb{Z}, 
\end{align*}
and a group homomorphism $\cl \colon K(C) \to \Gamma$
to be
\begin{align*}
\cl(E)=(\rank(E), \deg(E)). 
\end{align*}
Let $Z \colon \Gamma \to \mathbb{C}$ be
the map defined by 
\begin{align*}
Z(r, d)=-d+\sqrt{-1}r. 
\end{align*}
Then it is easy to see that $Z$ is a stability condition on 
$\Coh(C)$. 
An object $E\in \Coh(C)$ is $Z$-semistable if and only 
if $E$ is a torsion sheaf or $E$ is a semistable 
vector bundle
in the sense of Definition~\ref{sbund}. 

(ii) Let $A$ be a finite dimensional algebra
over $\mathbb{C}$ and $\aA$ the 
abelian category of 
finitely generated right $A$-modules. 
There is a finite number of simple
objects $S_1, S_2, \cdots, S_N$ 
in $\aA$ 
such that 
\begin{align*}
K(\aA) \cong \bigoplus_{i=1}^{N} \mathbb{Z}[S_i]. 
\end{align*}
We set $\Gamma =K(\aA)$ and $\cl=\mathrm{id}$. 
Choose elements, 
\begin{align*}
z_1, z_2, \cdots, z_N \in \mathbb{H}.
\end{align*}
Then the map 
$Z \colon \Gamma \to \mathbb{C}$
defined by 
\begin{align*}
Z\left(\sum_{i}a_i [S_i]\right) =\sum_{i}a_i z_i,
\end{align*}
is a stability condition on $\aA$. 

(iii) 
The following generalization of (i) 
will be used in the later sections. 
Let $X$ be a smooth projective 
variety over $\mathbb{C}$.
 We set
\begin{align*}
\Coh_{\le 1}(X) \cneq \{ E \in \Coh(X) :
\dim \Supp(E) \le 1 \}. 
\end{align*}
We set 
\begin{align*}
\Gamma_0 \cneq \mathbb{Z} \oplus H_2(X, \mathbb{Z}),
\end{align*}
and the group homomorphism 
$\cl_0 \colon K(\Coh_{\le 1}(X)) \to \Gamma_0$
to be 
\begin{align*}
\cl_0(E) \cneq (\ch_3(E), \ch_2(E)).
\end{align*}
By the Riemann-Roch theorem, $\cl_0(E)$ is also written as 
$(\chi(E), [E])$, 
where $[E]$ is the fundamental homology class 
determined by $E$ and $\chi(E)$ is the 
holomorphic Euler characteristic. 

Let $\omega$ be an 
$\mathbb{R}$-ample divisor on $X$.
We set $Z_{\omega} \colon \Gamma_0 \to \mathbb{C}$
to be
\begin{align*}
Z_{\omega}(n, \beta) \cneq 
-n +(\omega \cdot \beta) \sqrt{-1}. 
\end{align*}
Then $Z_{\omega}$ is a stability condition on 
$\Coh_{\le 1}(X)$. 
An object $E \in \Coh_{\le 1}(X)$ is 
$Z_{\omega}$-(semi)stable iff 
$E$ is $\omega$-Gieseker (semi)stable sheaf. 
(cf.~\cite{Hu}.)
If $\dim X=1$ and $\deg \omega=1$, then 
$Z_{\omega}$ coincides with
the stability condition constructed in 
(i). 
\end{exam}
\subsection{Wall-crossing phenomena}\label{subsec:Wall}
Here we explain a
rough idea of wall-crossing phenomena and 
a simple example. 
We set 
\begin{align*}
\Stab(\aA) \cneq \{ Z \in \Gamma_{\mathbb{C}}^{\vee} : 
Z \mbox{ is a stability condition on }\aA \}.
\end{align*}
For instance in Example~\ref{exam:stab} (ii), 
we have the identification, 
\begin{align*}
\Stab(\aA) \cong \mathbb{H}^N. 
\end{align*}
For $v\in \Gamma$, we are interested in 
`counting invariants', 
\begin{align*}
\Stab(\aA) \ni Z \mapsto I_{v}(Z) \in \mathbb{Q},
\end{align*}
where $I_v(Z)$ `counts' $Z$-semistable 
objects $E \in \aA$ with $\cl(E)=v$. 
There may be several choices of the 
definition of $I_v(Z)$. 
For instance we can consider moduli space of 
$Z$-semistable objects $E \in \aA$
with $\cl(E)=v$, denoted by $M_v(Z)$, and 
take $I_v(Z)$ to be 
\begin{align*}
I_v(Z)=\chi(M_v(Z)). 
\end{align*}
Here $\chi(\ast)$ is the topological Euler 
characteristic.
We need to check that the existence of the 
moduli space $M_v(Z)$, 
but this holds in the
cases given in Example~\ref{exam:stab}. 

In principle, there should be a wall and chamber
structure on the space $\Stab(\aA)$ such 
that $I_v(Z)$ is constant on a chamber but jumps 
on a wall. The set of walls is given by a
countable number of real codimension one submanifolds
$\{W_{\lambda}\}_{\lambda \in \Lambda}$ 
in $\Stab(\aA)$, and a chamber is a 
connected component, 
\begin{align*}
\cC \subset \Stab(\aA) \setminus \bigcup_{\lambda \in \Lambda}
W_{\lambda}. 
\end{align*}
For instance, let us consider the algebra $A$ given by 
\begin{align*}
A=
\begin{pmatrix}
\mathbb{C} & \mathbb{C} \\
0 & \mathbb{C}
\end{pmatrix}.
\end{align*}
Let $\aA$ be the abelian category 
of finitely generated right $A$-modules.
(In other words, $\aA$ is the category 
of representations of a quiver with 
two vertex and one arrow.) 
There are two simple objects
in $\aA$,  
\begin{align*}
S_i =\mathbb{C} \cdot e_i, \quad i=1, 2, 
\end{align*}
whose right $A$-actions are given by 
\begin{align*}
e_i \cdot \begin{pmatrix}
a_1 & a_3 \\
0 & a_2
\end{pmatrix}
=a_i e_i. 
\end{align*}
We take an object
 $E \in \aA$, 
which is isomorphic to $\mathbb{C}^2$
as a $\mathbb{C}$-vector space, and 
the right $A$-action is the standard one. 
There is an exact sequence in $\aA$, 
\begin{align}\label{SES}
0 \to S_2 \to E \to S_1 \to 0. 
\end{align}
Let us identify $\Stab(\aA)$
with $\mathbb{H}^2$, as in Example~\ref{exam:stab} (ii). 
For a stability condition 
\begin{align*}
Z=(z_1, z_2) \in \Stab(\aA)\cong \mathbb{H}^2,
\end{align*} 
the exact sequence (\ref{SES}) easily implies the 
following. 
\begin{align*}
E \mbox{ is }
\left\{
\begin{array}{cl}
Z\mbox{-stable} & \mbox{ if }\arg z_2 <\arg z_1 \\
Z\mbox{-semistable} & \mbox{ if }\arg z_2=\arg z_1 \\
\mbox{not }Z\mbox{-semistable} & \mbox{ if }\arg z_2>\arg z_1
\end{array}
  \right. 
\end{align*}
In particular for
an element
\begin{align*}
v=\cl(E)=(1, 1) \in \Gamma,
\end{align*}
the moduli space $M_v(Z)$ is 
\begin{align*}
M_v(Z)=\left\{
\begin{array}{cl}
\{E\} & \mbox{ if }\arg z_2 <\arg z_1 \\
\{E \} \cup \{S_1 \oplus S_2\}
& \mbox{ if }\arg z_2=\arg z_1 \\
\emptyset & \mbox{ if }\arg z_2>\arg z_1
\end{array}
  \right. 
\end{align*}
The `counting invariant'
$I_v(Z)=\chi(M_v(Z))$ is 
\begin{align*}
I_v(Z)=\left\{
\begin{array}{cl}
1 & \mbox{ if }\arg z_2 <\arg z_1 \\
2
& \mbox{ if }\arg z_2=\arg z_1 \\
0 & \mbox{ if }\arg z_2>\arg z_1
\end{array}
  \right. 
\end{align*}
Here we have observed wall-crossing phenomena 
of $I_v(Z)$, 
whose wall is given by 
\begin{align*}
W=\{(z_1, z_2) \in \mathbb{H}^2 : 
\arg z_1 =\arg z_2 \}. 
\end{align*}

\subsection{Weak stability conditions}\label{subsec:Weak}
A slightly generalized notion of 
stability conditions is sometimes 
useful. For instance if we consider 
stability conditions in the sense of 
Definition~\ref{def:A},
then there is no stability condition 
on $\Coh(X)$ if $\dim X \ge 2$. 
(cf.~\cite[Lemma~2.7]{Tolim}.) 
On the other hand, there 
are classical notions of 
stability conditions on 
$\Coh(X)$, such as slope stability.
(cf.~\cite{Hu}.) 
The slope stability can be formulated 
in the language of weak stability 
conditions introduced in~\cite{Tcurve1}. 

Let $\aA$ be an abelian category. 
As in Subsection~\ref{subsec:definition},  
we fix a finitely generated 
free abelian group $\Gamma$ together 
with a group homomorphism 
$\cl \colon K(\aA) \to \Gamma$. 
We also fix a filtration of $\Gamma$, 
\begin{align*}
0=\Gamma_{-1} \subsetneq
\Gamma_0 \subsetneq \Gamma_1 \subsetneq \cdots \subsetneq \Gamma_N
=\Gamma, 
\end{align*}
such that each subquotient $\Gamma_i/\Gamma_{i-1}$ is a 
free abelian group. 
\begin{defi}\label{def:wstab}
A weak stability condition on $\aA$ is
\begin{align*}
Z=\{ Z_i \}_{i=0}^{N} \in 
\prod_{i=0}^{N} \Hom_{\mathbb{Z}}(\Gamma_i/\Gamma_{i-1}, 
\mathbb{C}),
\end{align*}
such that the following conditions are 
satisfied: 

(i) For non-zero $E\in \aA$, take $-1 \le i \le N$
such that $\cl(E) \in \Gamma_i \setminus \Gamma_{i-1}$. 
(We regard $\Gamma_{-2}=\emptyset$.)
Then we have 
\begin{align*}
Z(E) \cneq Z_i([\cl(E)]) \in \mathbb{H}. 
\end{align*}
Here $[\cl(E)]$ is the class of 
$\cl(E)$ in $\Gamma_i/\Gamma_{i-1}$. 
We say $E\in \aA$ is $Z$-(semi)stable 
if for any exact sequence
$0 \to F \to E \to G \to 0$
in $\aA$, we have the inequality, 
\begin{align}\label{ineq:ZFG}
\arg Z(F) <(\le) \arg Z(G). 
\end{align}

(ii) There is a Harder-Narasimhan filtration 
for any $E\in \aA$. 
\end{defi}
When $N=0$, a weak 
stability conditions is a 
stability condition in the sense of 
Definition~\ref{def:A}.
\begin{rmk}
If the inequality (\ref{ineq:ZFG})
is strict, we have the following 
three possibilities: 
\begin{align}
\label{ZFEG1}
&\arg Z(F) < \arg Z(E) < \arg Z(G), \\
\label{ZFEG2}
&\arg Z(F) < \arg Z(E) = \arg Z(G), \\
\label{ZFEG3}
&\arg Z(F)= \arg Z(E) < \arg Z(G). 
\end{align} 
When $N=0$, i.e. $Z$ is a stability 
condition, then only the inequality (\ref{ZFEG1})
is possible. On the other hand when $N>0$, the 
inequalities (\ref{ZFEG2}), (\ref{ZFEG3}) 
are also possible. 
\end{rmk}
Here we give some examples.  
\begin{exam}\label{hereexam}
(i) Let $X$ be a $d$-dimensional 
smooth projective variety
and $\aA=\Coh(X)$. 
Take $\Gamma=\Imm \ch$, $\cl=\ch$
as in (\ref{cl=ch})
 and 
take a filtration 
\begin{align*}
\Gamma_0 \subset \Gamma_1 \subset \cdots \subset \Gamma_{d}, 
\end{align*}
given by 
\begin{align*}
\Gamma_i=\Gamma \cap H^{\ge 2d-2i}(X, \mathbb{Q}). 
\end{align*}
Choose 
\begin{align*}
0<\phi_d <\phi_{d-1} < \cdots <\phi_0<1
\end{align*}
and an ample divisor $\omega$ on $X$. 
Set $Z_i \colon \Gamma_i/\Gamma_{i-1} \to \mathbb{C}$
to be
\begin{align*}
Z_i(v)=\exp(\sqrt{-1}\pi \phi_i)\int_{X}v \cdot \omega^i. 
\end{align*}
Then $Z=\{Z_i\}_{i=0}^{d}$
is a weak stability condition on $\Coh(X)$. 
In this case, $E\in \Coh(X)$ is $Z$-semistable 
if and only if it is pure sheaf, i.e. 
there is no $0\neq F \subset E$ with 
$\dim \Supp(F)<\dim \Supp(E)$. 

(ii) Let $X$ be a smooth projective surface and 
take $\Gamma$ and $\cl$ as above. We set
$\Gamma_{0} \subset \Gamma_{1}=\Gamma$ to be
\begin{align*}
\Gamma_0=\Gamma \cap H^4(X, \mathbb{Q}), 
\end{align*}
hence 
\begin{align*}
\Gamma_1/\Gamma_0=\Gamma \cap (H^0 \oplus H^2). 
\end{align*}
We set $Z_i \colon \Gamma_i/\Gamma_{i-1} \to \mathbb{C}$
to be
\begin{align*}
Z_0(n)&=-n, \\
Z_1(r, D) &=-D\cdot \omega +\sqrt{-1}r. 
\end{align*}
Then $Z=\{Z_{i}\}_{i=0}^{1}$ is a 
weak stability condition
on $\Coh(X)$. 
An object $E\in \Coh(X)$ is $Z$-semistable if 
and only if $E$ is a torsion sheaf or an
$\omega$-slope semistable sheaf. 
(cf.~\cite{Hu}.)
\end{exam}

In~\cite{Tcurve1}, the space of weak 
stability conditions on triangulated categories is
introduced. 
Namely a weak stability condition on 
a triangulated category $\dD$ is 
a pair of $(Z, \aA)$, where $\aA$ is the heart 
of a bounded t-structure on $\dD$ and 
$Z$ is a weak stability condition on $\aA$. 
We denote by 
\begin{align}\label{StabG}
\Stab_{\Gamma_{\bullet}}(\dD), 
\end{align}
the set of weak stability 
conditions on $\dD$, satisfying 
some good properties,
i.e. local finiteness, support property.
See~\cite[Section~2]{Tcurve1} for the detail on 
these properties. Using the same 
argument by Bridgeland~\cite[Theorem~7.1]{Brs1}, it 
is proved in~\cite[Theorem~2.15]{Tcurve1}
that
the set (\ref{StabG}) has a natural topology and 
each connected component is a complex manifold.

\section{Curve counting invariants on Calabi-Yau 3-folds}\label{sec:Curve}
In this section, we recall several curve counting 
theories on Calabi-Yau 3-folds, 
conjectures and the results. 
 In what follows,
we call 
a smooth projective complex 3-fold \textit{Calab-Yau} if 
it satisfies the following condition, 
\begin{align*}
\bigwedge^{3} T_{X}^{\vee} \cong \oO_X, \quad
H^1(X, \oO_X)=0.  
\end{align*} 
For instance, the quintic 3-fold, 
\begin{align*}
X=\{ x_0^5 +x_1^5 +x_2^5 +x_3^5 +x_4^5=0\} 
\subset \mathbb{P}^4,
\end{align*}
is a famous example of a Calabi-Yau 3-fold. 
\subsection{Gromov-Witten theory}
Let $X$ be a smooth projective Calabi-Yau 3-fold
and $C$ a connected 1-dimensional reduced 
$\mathbb{C}$-scheme with at worst nodal 
singularities. A morphism of schemes 
\begin{align*}
f \colon C \to X,
\end{align*}
is a \textit{stable map}
if the set of isomorphisms $\phi \colon C \stackrel{\sim}{\to} C$
satisfying $f \circ \phi =f$ is a finite set. 
This condition is equivalent to one of the following conditions. 
\begin{itemize}
\item For any ample line bundle $\lL$ on $X$, the
line bundle $\omega_C \otimes f^{\ast} \lL^{\otimes 3}$
is an ample line bundle on $C$. Here $\omega_C$ is 
the dualizing sheaf of $C$. 
\item If $C' \subset C$ is an irreducible component such that $f(C')$
is a point, then 
\begin{align*}
2g(C') + \sharp \left( C' \cap (\overline{C\setminus C'}) \right) \ge 3.
\end{align*}
\end{itemize}
Here $g(\ast)$ is the arithmetic genus. 
The moduli space of such maps is constructed 
after we fix the following numerical data, 
\begin{align*}
g\in \mathbb{Z}_{\ge 0}, \quad \beta \in H_2(X, \mathbb{Z}). 
\end{align*}
We call a stable map $(C, f)$ as \textit{type $(g, \beta)$}
if $g(C)=g$ and 
the map $f$ satisfies $f_{\ast}[C]=\beta$. 
The moduli space of stable maps $(C, f)$ 
of type $(g, \beta)$ is denoted by,  
\begin{align}\label{moduli:map}
\overline{M}_{g}(X, \beta).
\end{align}
The moduli space (\ref{moduli:map})
is a Deligne Mumford stack of finite type over $\mathbb{C}$~\cite{Ktor}. 
However the space (\ref{moduli:map}) may be singular 
and its dimension may be different from its expected dimension. 
In fact the tangent space and the obstruction space
 of the space of maps 
$f \colon C \to X$ for a fixed $C$ are given by 
\begin{align*}
H^0(C, f^{\ast}T_X), \quad H^1(C, f^{\ast}T_X), 
\end{align*}
respectively. 
Hence the expected dimension of the space (\ref{moduli:map}) 
is 
\begin{align*}
&\chi(C, f^{\ast}T_X) +\dim \overline{M}_g \\
&=\frac{3}{2} \deg T_C +3g-3 \\ 
&=0. 
\end{align*}
Here $\overline{M}_g$ is the moduli 
space of genus $g$ stable curves. 
Here we have used the Riemann-Roch theorem
on $C$ and 
the Calabi-Yau assumption of $X$. 

Now there is a way to construct the 0-dimensional 
virtual fundamental cycle on (\ref{moduli:map}) 
via perfect obstruction theory~\cite{BF}, \cite{LTV}. 
By definition, a \textit{perfect obstruction theory} on a
scheme (or Deligne-Mumford stack) $M$ is a 
morphism in the derived category
of coherent sheaves $D^b \Coh(M)$,  
\begin{align}\label{pft}
 h \colon E^{\bullet} \to L_{M}, 
\end{align}
where $E^{\bullet}$ is a complex 
of vector bundles on $M$
concentrated on $[-1, 0]$ and $L_{M}$ is the cotangent 
complex of $M$. 
The morphism $h$ should satisfy that  
$h^0$ is an isomorphism and $h^{-1}$ is surjective. 
Given such a morphism (\ref{pft}), 
we are able to construct the 
virtual fundamental cycle, 
\begin{align*}
[M]^{\rm{vir}}  \in  A_{\rank E^0 -\rank E^{-1}}(M).   
\end{align*}
Here $A_{\ast}(M)$ is the Chow group 
of $M$. 
Roughly speaking, the cycle 
$[M]^{\rm{vir}}$ is constructed 
by taking the intersection of the intrinsic normal 
cone and the 0-section in the 
vector bundle stack $[(E^{-1})^{\vee}/(E^{0})^{\vee}]$. 
(See~\cite{BF}, \cite{LTV} for the detail.)

By~\cite{BF}, \cite{LTV}, there is a perfect
obstruction theory on the moduli space (\ref{moduli:map}). 
The resulting virtual fundamental cycle is denote by
\begin{align*}
[\overline{M}_{g}(X, \beta)]^{\rm{vir}} \in 
A_{0}(\overline{M}_g(X, \beta), \mathbb{Q}). 
\end{align*}
Integrating the virtual cycle, we obtain the 
GW invariant. 
\begin{defi}\emph{
The \textit{Gromov-Witten (GW) invariant} is defined by 
\begin{align*}
N_{g, \beta}^{\rm{GW}}= \int_{[\overline{M}_{g}(X, \beta)]^{\rm{vir}}} 1 \in \mathbb{Q}. 
\end{align*}}
\begin{rmk}
Since $\overline{M}_g(X, \beta)$ is not a scheme but 
a Deligne-Mumford stack, the resulting invariant 
$N_{g, \beta}^{\rm{GW}}$ is not an integer
in general. 
\end{rmk}
\end{defi}
One of the important examples is 
a contribution of multiple covers 
to a fixed super rigid rational curve. 
\begin{exam}\label{GW:rigid}
Let 
\begin{align*}
f \colon X \to Y,
\end{align*}
be a birational contraction which 
contracts a smooth super rigid 
rational curve $C\subset X$, i.e. 
\begin{align*}
N_{C/X}=\oO_{\mathbb{P}^1}(-1) \oplus \oO_{\mathbb{P}^1}(-1).
\end{align*}  
In this case, the computation of $N_{g, d[C]}^{\rm{GW}}$
can be reduced to a certain integration over
the space $\overline{M}_{g}(\mathbb{P}^1, d)$. 
We have the following diagram:
\begin{align*}
\xymatrix{
\cC \ar[r]^{\phi} \ar[d]_{\pi}  & \mathbb{P}^1, \\
 \overline{M}_{g}(\mathbb{P}^1, d)  &
}
\end{align*}
where $\pi$ is 
the universal curve and
$\phi$ is the universal morphism. 
Then we have 
\begin{align}\label{exam:-1-1}
N_{g, d[C]}^{\rm{GW}} = \int_{[{\overline{M}_{g}(\mathbb{P}^1, d)]^{\rm{vir}}}}
c_{\rm{top}}(R^1 \pi_{\ast} \phi^{\ast}\oO_{\mathbb{P}^1}(-1)^{\oplus 2}). 
\end{align}
The invariants (\ref{exam:-1-1}) are
 computed in~\cite{FaPan}, 
\begin{align*}
N_{0, d[C]}^{\rm{GW}}&=\frac{1}{d^3}, \quad N_{1, d[C]}=\frac{1}{12d}, \\ 
N_{g, d[C]}^{\rm{GW}}&=\frac{\lvert B_{2g} \rvert \cdot d^{2g-3}}{2g \cdot (2g-2)!}, \ 
g \ge 2. 
\end{align*}
Here $B_{2g}$ is the 2g-th Bernoulli number. 
\end{exam}

\subsection{Donaldson-Thomas theory}
Another curve counting invariant on 
a Calabi-Yau 3-fold $X$ is defined by 
the integration of the virtual fundamental cycle on 
the moduli space of subschemes, 
\begin{align}\label{sub}
Z \subset X, 
\end{align}
satisfying $\dim Z \le 1$. 
Given a numerical data, 
\begin{align*}
n\in \mathbb{Z}, \quad \beta \in H_2(X, \mathbb{Z}), 
\end{align*}
the relevant moduli space is the classical 
Hilbert scheme, 
\begin{align}\label{Hilb}
\Hilb_n(X, \beta), 
\end{align}
which parameterizes subschemes (\ref{sub})
satisfying
\begin{align}\label{chZ}
\chi(\oO_Z)=n, \quad [Z]=\beta.
\end{align}
Recall that the moduli space (\ref{Hilb})
is a projective scheme. 

The moduli space (\ref{Hilb}) is also 
interpreted as a moduli space of rank one torsion 
free 
sheaves on $X$ with a trivial first 
Chern class. 
Namely if $I$ is a torsion free sheaf of rank one, 
then $I$ fits into the exact sequence, 
\begin{align*}
0 \to I \to I^{\vee \vee} \to F \to 0,
\end{align*}
such that $F$ is one or zero dimensional sheaf. 
It can be shown that 
$I^{\vee \vee}$ is a line bundle on $X$, 
hence isomorphic to $\oO_X$ if its first Chern class is
zero. Hence $I$ is
isomorphic to $I_Z$, the ideal 
sheaf of a subscheme $Z \subset X$
with $\dim Z \le 1$. 
The condition (\ref{chZ})
is equivalent to the condition on the Chern character, 
\begin{align}\label{chI}
\ch(I_Z) &=(1, 0, -\beta, -n)  \\
 &\in H^0(X, \mathbb{Z}) \oplus H^2(X, \mathbb{Z})
 \oplus H^4(X, \mathbb{Z})
\oplus H^6(X, \mathbb{Z}). 
\end{align}
Here we have regarded $\beta$ and $n$ as elements of 
$H^4(X, \mathbb{Z})$ and $H^6(X, \mathbb{Z})$ 
by the Poincar\'e duality.
As a summary, there is a one to one correspondence 
between subschemes (\ref{sub})
satisfying (\ref{chZ})
and torsion free sheaves $I$ on $X$ satisfying (\ref{chI}),
via $Z \mapsto I_Z$.  

If we regard the space (\ref{Hilb})
as a moduli space of rank one torsion free
 sheaves, 
the deformation theory of coherent sheaves 
implies that the spaces
\begin{align*}
\Ext_{X}^1(I_Z, I_Z), \quad \Ext_X^2(I_Z, I_Z), 
\end{align*}
are tangent space and the obstruction 
space at the point $[Z] \in \Hilb_{n}(X, \beta)$
respectively. 
Since $X$ is a Calabi-Yau 3-fold, the 
Serre duality implies that 
\begin{align*}
\Ext_{X}^2(I_Z, I_Z) \cong \Ext_{X}^{1}(I_Z, I_Z)^{\vee}. 
\end{align*}
In particular the expected dimension of 
the space (\ref{Hilb}) is
\begin{align*}
\dim \Ext_{X}^1(I_Z, I_Z) - \dim \Ext_{X}^2(I_Z, I_Z)=0. 
\end{align*} 
In fact there is a perfect obstruction theory 
on $\Hilb_n(X, \beta)$, 
(cf.~\cite{Thom},) 
\begin{align*}
E^{\bullet} \to L_{\Hilb_n(X, \beta)}, 
\end{align*}
satisfying that 
\begin{align}\label{sym}
E^{\bullet} \cong E^{\bullet \vee}[1]. 
\end{align}
A perfect obstruction theory satisfying the 
symmetry (\ref{sym}) is called a \textit{perfect
symmetric obstruction theory}. 
We have the associated virtual fundamental cycle, 
\begin{align*}
[\Hilb_n(X, \beta)]^{\rm{vir}} \in A_0(\Hilb_n(X, \beta), \mathbb{Z}). 
\end{align*}
The DT invariant is defined by the integration
over the virtual fundamental cycle. 
\begin{defi}\emph{
The \textit{Donaldson-Thomas (DT) invariant}
is defined by 
\begin{align}\label{def:DT}
I_{n, \beta}=\int_{[\Hilb_n(X, \beta)]^{\rm{vir}}}1 \in \mathbb{Z}.
\end{align}
}
\end{defi}
So far, $I_{n, \beta}$ are computed in several examples
in terms of generating functions. 
\begin{exam}\label{ex:DT}
(i) In the case of $\beta=0$, 
the generating series of $I_{n, 0}$ is 
computed by Li~\cite{Li}, Behrend-Fantechi~\cite{BBr}
 and Levine-Pandharipande~\cite{LP}, 
\begin{align*}
\sum_{n\in \mathbb{Z}}I_{n, 0}q^n
=M(-q)^{\chi(X)}. 
\end{align*}
Here $M(q)$ is the MacMahon function, 
\begin{align*}
M(q) &=\prod_{k \ge 1}\frac{1}{(1-q^k)^k} \\
&=1+q+3q^2 +6q^3 + \cdots. 
\end{align*}
(ii) Let $C \subset X$ is a super rigid 
rational curve as in Example~\ref{GW:rigid}. 
Then the invariant $I_{n, d[C]}$ is computed by 
Behrend-Bryan~\cite{BeBryan}, 
\begin{align*}
\sum_{n, d}I_{n, d[C]}q^n t^{d}
=M(-q)^{\chi(X)}\prod_{k \ge 1}(1-(-q)^k t)^k. 
\end{align*}
\end{exam}

\subsection{DT theory via Behrend function}
The integration (\ref{def:DT})
is usually difficult to compute. On the other hand, 
Behrend~\cite{Beh} shows that the invariant (\ref{def:DT}) 
is also obtained as a certain weighted Euler 
characteristic of a certain constructible 
function on $\Hilb_n(X, \beta)$. 
In many situations, computations of 
weighted Euler characteristic are easier 
than computations of virtual fundamental cycles. 

In fact for any $\mathbb{C}$-scheme 
$M$, Behrend~\cite{Beh} constructs
a canonical constructible function, 
\begin{align*}
\nu_{M} \colon M \to \mathbb{Z},
\end{align*}
satisfying the following properties. 
\begin{itemize}
\item If $\pi \colon M_1 \to M_2$ is a smooth morphism 
with relative dimension $d$, we have 
\begin{align*}
\nu_{M_1}=(-1)^d \pi^{\ast} \nu_{M_2}. 
\end{align*}
\item For $p\in M$, suppose that there is an analytic 
open neighborhood $p\in U \subset M$, a complex manifold
$V$ and a holomorphic function $f\colon V \to \mathbb{C}$
such that $U \cong \{df=0 \}$. Then we have 
\begin{align}\label{nu}
\nu(p)=(-1)^{\dim V}(1-\chi(M_p(f))). 
\end{align}
Here $M_p(f)$ is the Milnor fiber of $f$ at $p\in V$. 
\item If $M$ has a symmetric perfect obstruction theory, we have 
\begin{align}\label{Mvir}
\int_{[M]^{\rm vir}}1 &=\int_{M} \nu_{M} d\chi, \\
\notag
&\cneq \sum_{k \in \mathbb{Z}} k \chi(\nu^{-1}(k)).  
\end{align}
\end{itemize}
Here the Milnor fiber $M_p(f)$ is defined as follows. 
Let $p\in V' \subset V$ be an analytic small neighborhood 
and fix a norm $\lVert \ast \rVert$ on $V'$. 
Then for $0<\varepsilon \ll \delta \ll 1$, the
topological type of the space
\begin{align}\label{Mil}
\{ z\in V' : \lVert z -p \rVert \le \delta, \
f(z)=f(p)+\epsilon \}, 
\end{align}
does not depend on $\varepsilon$, $\delta$.
The Milnor fiber $M_p(f)$ is defined to be the 
topological space (\ref{Mil}). 

By the property (\ref{Mvir}), the invariant $I_{n, \beta}$
is also obtained by 
\begin{align*}
I_{n, \beta}= \int_{\Hilb_n(X, \beta)} \nu d\chi. 
\end{align*}
Here we have written $\nu_{\Hilb_n(X, \beta)}$ 
as $\nu$ for simplicity. 
An important fact is that 
the local moduli space of objects
in $\Coh(X)$ is 
analytically locally  
written as a critical locus of some 
holomorphic function on a complex manifold
up to gauge equivalence. 
This fact is proved in~\cite[Theorem~5.2]{JS} in a more general 
setting. 
In particular 
the function $\nu$
on $\Hilb_n(X, \beta)$
 can be computed 
using the expression (\ref{nu}). 

A rough idea of the proof of the
critical locus condition 
in~\cite[Theorem~5.2]{JS}
is as follows: 
for $E \in \Coh(X)$, we are interested in 
the deformations of $E$. 
By applying spherical twists associated to 
line bundles, we may assume that $E$ 
is a locally free sheaf, or equivalently 
a holomorphic vector bundle. 
(cf.~\cite[Corollary~8.5]{JS}.)
Let 
\begin{align*}
\overline{\partial} \colon 
E \to E\otimes \Omega^{0, 1},
\end{align*}
be the $\overline{\partial}$-connection 
which determines a holomorphic structure 
of $E$, 
where $\Omega^{0, 1}$ is the sheaf of 
$(0, 1)$-forms of $X$. 
Then giving a deformation of $E$ is 
equivalent to giving a deformation of $\overline{\partial}$
up to gauge equivalence.
This is equivalent to giving 
\begin{align*}
A \in A^{0, 1}(X, \eE nd(E)),
\end{align*}
where $A^{0, 1}(X, \eE nd(E))$ is the space of 
$\eE nd(E)$-valued $(0, 1)$ forms, 
satisfying 
\begin{align}\label{partial}
(\overline{\partial} +A)^2 =0, 
\end{align}
up to gauge equivalence. 
The equation (\ref{partial}) is equivalent to 
\begin{align*}
\overline{\partial}A + A \wedge A=0. 
\end{align*}
Let $\mathrm{CS}$ be the 
\textit{holomorphic Chern Simons function}, 
\begin{align*}
\mathrm{CS} \colon A^{0, 1}(X, \eE nd(E))
\to \mathbb{C}, 
\end{align*}
defined by 
\begin{align*}
\mathrm{CS}(A) =\int_{X} \left( \frac{1}{2} \overline{\partial}A 
\wedge A +\frac{1}{3} A \wedge A \wedge A \right) \wedge
\sigma_X,
\end{align*}
where $\sigma_X$ is a no-where vanishing 
holomorphic 3-form on $X$. 
(cf.~\cite{Thom}.)
Then $A \in A^{0, 1}(X, \eE nd(E))$
satisfies the equation (\ref{partial})
if and only if $A$ is a critical locus of 
the function $\mathrm{CS}$. Therefore the local 
moduli space of $E$ is written as 
\begin{align*}
\{d\mathrm{CS}=0\}/G, 
\end{align*}
where $G$ is the group of isomorphisms of $E$
as a $C^{\infty}$-vector bundle, i.e. 
the local moduli space of objects
in $\Coh(X)$ is written as a critical locus
up to gauge equivalence. 

However $A^{0, 1}(X, \eE nd(E))$ is 
an infinite dimensional vector space, and we need to 
find suitable finite dimensional vector subspace of 
$A^{0, 1}(X, \eE nd(E))$. 
This is worked out in~\cite[Theorem~5.2]{JS} by using the 
Hodge theory.
Namely the space of harmonic forms
$U$ on 
$A^{0, 1}(X, \eE nd(E))$ is finite dimensional, 
satisfying $U\cong \Ext^1(E, E)$, and 
we restrict $\mathrm{CS}$ to $U$.  
For the detail, see~\cite[Theorem~5.2]{JS}.
\begin{exam}
(i) Suppose that $\Hilb_n(X, \beta)$ is non-singular
of dimension $d$. By the property (\ref{nu}), 
the Behrend function on $\Hilb_n(X, \beta)$
coincides with $(-1)^d$. Therefore we have 
\begin{align*}
I_{n, \beta}=(-1)^d \chi(\Hilb_n(X, \beta)). 
\end{align*}
(ii) Suppose that $\Hilb_n(X, \beta)$ is 
isomorphic to the spectrum of 
$\mathbb{C}[z]/z^{k}$ for some
$k\ge 1$. (For instance, the  
local moduli space of a rigid rational 
curve $C\subset X$ with $N_{C/X}=\oO_C \oplus 
\oO_C(-2)$ is written as
the spectrum of $\mathbb{C}[z]/z^k$ for some $k\ge 1$.)
Then $\Hilb_n(X, \beta)$ is 
written as $\{ df=0 \}$, where $f$ is
\begin{align*} 
f \colon \mathbb{C} \ni z \mapsto z^{k+1} \in \mathbb{C}. 
\end{align*}
The Milnor fiber of $f$ at $0 \in \mathbb{C}$
is $(k+1)$-points, hence we have 
\begin{align*}
I_{n, \beta}=\nu(0)=k. 
\end{align*}
\end{exam}

\subsection{GW/DT correspondence}
As we mentioned before, GW
invariant is not necessary an integer
while DT invariant is always an integer. 
Although both theories seem different, 
 Maulik-Nekrasov-Okounkov-Pandharipande~\cite{MNOP}
propose a conjecture on a certain
relationship between GW and DT theories. 
The conjecture is formulated  
in terms of generating functions, and
it also implies a hidden integrality of GW invariants. 

Let us introduce the generating functions. 
The generating function of GW side is 
\begin{align*}
\mathrm{GW}(X) =\sum_{g\ge 0, \beta>0}N_{g, \beta}^{\rm{GW}}\lambda^{2g-2}t^{\beta}. 
\end{align*}
Here $\beta>0$ means $\beta$ is a homology class of 
a non-zero effective one cycle on $X$. 
Similarly the generating function of DT side is 
\begin{align*}
\DT(X)=\sum_{n\in \mathbb{Z}, \beta \ge 0}I_{n, \beta}q^n t^{\beta}. 
\end{align*}
The series $\DT(X)$ can be written as 
\begin{align*}
\DT(X)=\sum_{\beta \ge 0} \DT_{\beta}(X) t^{\beta}, 
\end{align*}
where $\DT_{\beta}(X)$ is a Laurent series of $q$. 
(It is easy to check that $\Hilb_n(X, \beta)=\emptyset$, 
hence $I_{n, \beta}=0$, 
for $n\ll 0$.)
The term $\DT_0(X)$ is a contribution of 
zero dimensional subschemes, and does not 
contribute to curve counting on $X$. 
The \textit{reduced DT series}
are defined by 
\begin{align}\label{reduced}
\DT'(X) =\frac{\DT(X)}{\DT_0(X)}, \
\DT'_{\beta}(X) =\frac{\DT_{\beta}(X)}{\DT_0(X)}. 
\end{align}
Note that $\DT_0(X)$ is given by the power 
of the MacMahon function by Example~\ref{ex:DT} (i). 
\begin{conj}\emph{ \bf{\cite{MNOP}} }\label{conj:GWDT}

{\bf (i) (Rationality conjecture): }
The Laurent series $\DT_{\beta}'(X)$ is the 
Laurent expansion of a rational function of $q$, 
invariant under $q\leftrightarrow 1/q$. 

{\bf (ii) (GW/DT correspondence): }
By the variable change $q=-e^{i\lambda}$, we have
the equality of the generating series, 
\begin{align*}
\exp \mathrm{GW}(X)=\DT'(X). 
\end{align*} 
\end{conj}
Here we need some explanation on the 
above conjecture.
The series $\DT_{\beta}'(X)$ is a priori 
a Laurent series of $q$ and 
it is not obvious whether it converges or not 
near $q=0$. The rationality conjecture 
asserts that $\DT_{\beta}'(X)$
actually converges near $q=0$, and 
moreover it can be analytically continued 
to give a meromorphic function
(in fact rational function) on the 
$q$-plane. The invariance
under $q\leftrightarrow 1/q$
implies that the above analytic continuation   
satisfies the automorphic property with respect 
the transformation $q\leftrightarrow 1/q$. 
For instance in the situation of Example~\ref{GW:rigid}, 
the series $\DT_{[C]}'(X)$ is 
\begin{align}\notag
\DT_{[C]}'(X) &=q-2q^2 +3q^3- \cdots, \\
\label{DTC}
&= \frac{q}{(1+q)^2}.
\end{align}
The rational function (\ref{DTC})
is invariant under $q \leftrightarrow 1/q$. 

If we assume the rationality conjecture, 
we can expand $\DT'(X)$ near $q=-1$, and 
write it by the $\lambda$-variable 
via $q=-e^{i\lambda}$.
 The invariance 
of $\DT_{\beta}'(X)$ under $q\leftrightarrow 1/q$
implies
 that $i$ is not involved in the $\lambda$-expansion. 
The GW/DT correspondence asserts that the coefficients 
of the above expansion are described in terms of GW invariants. 

So far the above conjecture has been 
checked in several situations. 
For instance the GW/DT correspondence 
for a local $(-1, -1)$-curve
can be checked from Example~\ref{GW:rigid} and Example~\ref{ex:DT}, 
as discussed in~\cite{BeBryan}. 
Also GW/DT correspondence for toric Calabi-Yau 3-folds
and local curves are proved in~\cite{MNOP} and~\cite{OkPa} respectively, 
by using torus localization and degeneration formula.  
On the other hand,
at this moment, these arguments are 
applied to the above specific examples, and 
not to arbitrary Calabi-Yau 3-folds. 
We have few tools in approaching 
 Conjecture~\ref{conj:GWDT} 
in a general setting, 
except the recent progress of 
\textit{wall-crossing formula}
of DT type invariants. 
This is 
established by Joyce-Song~\cite{JS} and Kontsevich-Soibelman~\cite{K-S}, 
and 
is an effective tool in studying DT type 
curve counting invariants for arbitrary Calabi-Yau 
3-folds. 
So far, several applications have been given, including 
Conjecture~\ref{conj:GWDT} (i). 

A rough idea of the application of the
 wall-crossing formula
 is as follows.
Recall that the moduli space $\Hilb_n(X, \beta)$
is interpreted as a moduli space of torsion free
rank one sheaves on $X$. 
This is nothing but the moduli space of 
stable objects on $\Coh(X)$ w.r.t. weak 
stability conditions in Example~\ref{hereexam} (i). 
One may try to change weak stability conditions on 
$\Coh(X)$, construct other DT type invariants
counting stable objects,
and see wall-crossing phenomena as we discussed
in Subsection~\ref{subsec:Wall}. 
 However we can easily see that 
 there is no interesting 
wall-crossing phenomena w.r.t weak stability 
conditions constructed in Example~\ref{hereexam} (i).
Instead we  
can also study (weak) stability conditions on 
other abelian subcategory in the derived category 
of coherent sheaves $D^b \Coh(X)$,
e.g. the heart of a bounded t-structure on $D^b \Coh(X)$.
Then we can construct DT type invariants
counting stable objects in the derived category, 
and the wall-crossing formula  
describes how these invariants vary under change of 
(weak) stability conditions. 
If we choose some specific (weak) stability condition, 
then the generating series sometimes becomes simpler 
than the original DT series, thus giving 
some non-trivial result to the DT series. 

As we mentioned, an important point is 
that the wall-crossing formula is applied for any 
Calabi-Yau 3-fold, and not restricted to 
specific examples, e.g. toric Calabi-Yau 3-folds. 
Using this new kind of technology, Conjecture~\ref{conj:GWDT} (i) 
is now solved.\footnote{We need the result of~\cite{BG}
which is not yet written at this moment.}
 We will discuss this more in Subsection~\ref{subsec:product}
below. 

\subsection{Pandharipande-Thomas theory}\label{subsec:PT}
Another application of the wall-crossing formula 
is the so called DT/PT correspondence, 
that is the correspondence between 
DT invariants and invariants counting 
stable pairs~\cite{PT}.
The notion of stable pairs is introduced by
Pandharipande-Thomas~\cite{PT}
in order to give a geometric understanding of
the reduced DT theory (\ref{reduced}). By definition, a 
\textit{stable 
pair} on a Calabi-Yau 3-fold $X$ is a pair 
\begin{align*}
(F, s), 
\end{align*}
where $F$ is a coherent sheaf 
on $X$
and $s \colon \oO_X \to F$ is a morphism satisfying 
the following. 
\begin{itemize}
\item $F$ is a pure one dimensional sheaf, 
i.e. there is no zero dimensional subsheaf 
in $F$. 
\item The cokernel of $s$ is a zero dimensional sheaf. 
\end{itemize}
For instance let $C\subset X$ be a smooth 
curve and $D \subset C$ a divisor on $C$. 
We set $F=\oO_C(D)$ and define the morphism
 $s$ to be the composition, 
\begin{align*}
s \colon \oO_X \twoheadrightarrow \oO_C \hookrightarrow \oO_C(D). 
\end{align*}
Then the pair $(F, s)$ is a stable pair. 
As the above example indicates, 
roughly speaking, a stable pair is a 
pair of a 
curve on $X$ and an effective divisor on it. 

Note that if $Z \subset X$ is a subscheme 
giving a point in $\Hilb_n(X, \beta)$, 
we have a pair
\begin{align*}(\oO_Z, s), \quad s \colon 
\oO_X \twoheadrightarrow \oO_Z, 
\end{align*} 
where $s$
is a natural surjection. 
The pair $(\oO_Z, s)$ fails to be a stable 
pair if and only if $\oO_Z$ contains a
zero dimensional subsheaf. On the 
other hand, a stable pair $(F, s)$ 
determines a point in $\Hilb_n(X, \beta)$
if and only if $s$ is surjective. 

Similarly to the DT theory, we consider the moduli 
space of stable pairs $(F, s)$
satisfying 
\begin{align*}
[F]=\beta, \quad \chi(F)=n. 
\end{align*}
Here $[F]$ is the fundamental homology 
class determined by the one dimensional sheaf $F$. 
The resulting moduli space is 
denoted by 
\begin{align}\label{moduli:PT}
P_n(X, \beta).
\end{align}
The moduli space (\ref{moduli:PT}) is proved to be 
a
projective scheme in~\cite{PT}. 
Moreover the space (\ref{moduli:PT}) is 
interpreted as a moduli space of two term complexes, 
\begin{align}\label{twoterm}
I^{\bullet}=\cdots \to 
0 \to \oO_X \stackrel{s}{\to} F \to 0 \to \cdots, 
\end{align}
in the derived category of coherent sheaves, i.e. 
\begin{align*}
I^{\bullet} \in D^b \Coh(X). 
\end{align*}
The deformation theory of objects in the 
derived category yields that the spaces
\begin{align*}
\Ext_{X}^{1}(I^{\bullet}, I^{\bullet}), \quad 
\Ext_{X}^2(I^{\bullet}, I^{\bullet}),
\end{align*}
are tangent space and the obstruction space 
respectively, which are dual by the Serre duality. 
Similarly to the DT theory, the above 
deformation theory provides a perfect 
symmetric obstruction theory on
the space (\ref{moduli:PT}), hence 
the 0-dimensional virtual cycle. 
\begin{defi}\emph{
The \textit{Pandharipande-Thomas (PT) invariant} is defined by 
\begin{align*}
P_{n, \beta}=\int_{[P_n(X, \beta)]^{\rm{vir}}}1 \in \mathbb{Z}. 
\end{align*}
}
\end{defi}
As in the DT case, the invariant $P_{n, \beta}$ 
is also defined by 
\begin{align*}
P_{n, \beta}=\int_{P_n(X, \beta)}\nu d\chi,
\end{align*}
for the Behrend function, 
\begin{align*}
\nu \colon P_n(X, \beta) \to \mathbb{Z}. 
\end{align*}
\begin{exam}\label{ex:DTPT}
Let $C\cong \mathbb{P}^1 \subset X$ be
a super rigid rational curve as in Example~\ref{GW:rigid}. 
Then $(F, s)$ is a stable pair with $[F]=[C]$
and $\chi(F)=n$ if and only if
\begin{align*}
F=\oO_C(n-1), \quad 
s\in H^0(C, \oO_C(n-1)) \setminus \{0\}. 
\end{align*}
Hence we have, 
\begin{align*}
P_n(X, [C]) &\cong \mathbb{P}(H^0(C, \oO_C(n-1))), \\
 &\cong \mathbb{P}^{n-1}.
\end{align*}
Therefore we have 
\begin{align*}
P_{n, [C]} &=(-1)^{\dim P_n(X, [C])} \chi(P_n(X, [C])), \\
           &=(-1)^{n-1} n.
\end{align*}
The generating series is 
\begin{align*}
\sum_{n\in\mathbb{Z}}P_{n, [C]}q^n =
q-2q^2 +3q^3 -\cdots. 
\end{align*}
Note that the above series coincides with $\DT_{[C]}'(X)$ by (\ref{DTC}). 
\end{exam}
Similarly to the DT theory, we 
consider the generating series, 
\begin{align*}
\PT(X) &=\sum_{n\in \mathbb{Z}, \beta \ge 0}
P_{n, \beta}q^n t^{\beta} \\
 &=1+\sum_{\beta>0} \PT_{\beta}(X),
\end{align*}
where $\PT_{\beta}(X)$ is a Laurent series of $q$. 
In~\cite{PT}, Pandharipande-Thomas propose the following 
conjecture. 
\begin{conj}\label{conj:DTPT}
We have the equality of the generating series, 
\begin{align}\label{form:DTPT}
\DT'_{\beta}(X)=\PT_{\beta}(X). 
\end{align}
\end{conj}
Note that
we have already observed the formula (\ref{form:DTPT}) 
in Example~\ref{ex:DTPT} when the curve class is a 
class of a super
rigid rational curve. 

Similarly to Conjecture~\ref{conj:GWDT} (i), 
the formula (\ref{form:DTPT}) is also 
a consequence of the wall-crossing formula. 
A rough idea is as follows. Suppose that there is 
an abelian subcategory $\aA$ in $D^b \Coh(X)$ and a 
stability condition $\sigma$ on it,
 such that the ideal sheaf 
$I_Z$ for a 1-dimensional subscheme $Z\subset X$ is 
a $\sigma$-stable object in $\aA$. 
If there is a 0-dimensional subsheaf 
$Q \subset \oO_Z$, i.e. $\oO_X \to \oO_Z$ is 
not a stable pair, then there is a sequence, 
\begin{align}\label{QII}
Q[-1] \to I_Z \to I_{Z'},
\end{align}
where $Z'$ is a 1-dimensional subscheme in $Z$ defined by 
$\oO_{Z'}=\oO_Z/Q$. 
Suppose that the sequence (\ref{QII}) is an exact sequence in $\aA$. 
Then we expect that we can deform 
a stability condition $\sigma$ to another 
stability condition $\tau$ such that the sequence (\ref{QII})
destabilizes $I_Z$ w.r.t. $\tau$. Instead if we take
an exact sequence in $\aA$,  
\begin{align*}
I_{Z'} \to E \to Q[-1], 
\end{align*}
then the object $E$ may be $\tau$-stable. 
Such an object $E$ is isomorphic to a two term complex, 
\begin{align*}
E \cong (\oO_X \stackrel{s}{\to} F), 
\end{align*}
for a one dimensional sheaf $F$, 
and one may expect that $(F, s)$
is a stable pair. 
If the above story is correct, then 
 $\sigma$ corresponds to the
DT theory, $\tau$ corresponds to the PT theory, and the relationship between 
these theories should be described by the wall-crossing formula.

\subsection{Product formula of the generating series}
\label{subsec:product}
In this subsection, 
 we discuss the 
result
obtained by applying the wall-crossing formula. 
\begin{thm}
\emph{\bf{\cite{Tolim2}, \cite{Tcurve1}, \cite{BrH}}}\label{main:thm}
For each $n\in \mathbb{Z}$ and $\beta \in H_2(X, \mathbb{Z})$, 
there are invariants, 
\begin{align*}
N_{n, \beta} \in \mathbb{Q}, \quad L_{n, \beta} \in \mathbb{Q}, 
\end{align*}
satisfying that 
\begin{itemize}
\item there is $d\in \mathbb{Z}_{>0}$ such that 
$N_{n, \beta}=N_{n', \beta}$ if $n\pm n' \in d \mathbb{Z}$ 
and $\beta \neq 0$. 
\item $L_{n, \beta}=L_{-n, \beta}$, and 
$L_{n, \beta}=0$ for $\lvert n \rvert \gg 0$,
\end{itemize}
such that we have the following infinite 
product expansion formula, 
\begin{align} \label{product1}
\PT(X)&=\prod_{n>0, \beta >0}
\exp\left((-1)^{n-1} nN_{n, \beta}q^n t^{\beta}\right)
\left(\sum_{n, \beta}L_{n, \beta}q^n t^{\beta} \right), \\
\label{product2}
\DT(X)&= \prod_{n>0}
\exp\left((-1)^{n-1}n N_{n, 0}q^n  \right) \PT(X). 
\end{align}
\end{thm}
We will explain how to deduce the formula (\ref{product1}) 
via wall-crossing in Section~\ref{rational}. 
\begin{rmk}\label{rmk:BG}
More precisely, the results in~\cite{Tolim2}, \cite{Tcurve1}
are Euler characteristic versions of the corresponding 
results, i.e. take the (non-weighted) Euler characteristic
in defining the invariants $I_{n, \beta}$, $P_{n, \beta}$. 
As discussed in the arXiv version of~\cite[Theorem~8.11]{Tcurve1}, 
the formulas (\ref{product1}), (\ref{product2})
can be proved by combining the work of Joyce-Song~\cite{JS}
and Behrend-Getzler's announced result~\cite{BG}. 
The latter result is the derived category version 
of~\cite[Theorem~5.3]{JS}, 
that is the moduli stack of certain objects in the 
derived category is locally written as a 
critical locus of some holomorphic function up
to gauge action. The precise statement is 
formulated in~\cite[Conjecture~4.3]{Tcurve2}. 
On the other hand, in~\cite{BrH}, Bridgeland 
proves Theorem~\ref{main:thm} without relying~\cite{BG}, 
using the arguments different to ours. 
\end{rmk}
The invariants $N_{n, \beta}$ and $L_{n, \beta}$
are also interpreted as counting invariants of 
certain objects in the derived category. 
Roughly speaking: 
\begin{itemize}
\item 
Let $\omega$ be an $\mathbb{R}$-ample divisor 
and $Z_{\omega}$ the stability condition 
on $\Coh_{\le 1}(X)$
constructed in Example~\ref{exam:stab} (iii). 
In the notation of Example~\ref{exam:stab} (iii), 
the invariant $N_{n, \beta}$ counts 
$Z_{\omega}$-semistable objects $E \in \Coh_{\le 1}(X)$, 
satisfying 
\begin{align*}
\cl_0(E)=(n, \beta) \in \Gamma_0. 
\end{align*}
\item The invariant $L_{n, \beta}$
counts certain semistable objects in the derived category
$E\in D^b \Coh(X)$,
satisfying 
\begin{align*}
\ch(E) &=(1, 0, -\beta, -n)  \\
&\in H^0(X, \mathbb{Z}) \oplus H^2(X, \mathbb{Z})
 \oplus H^4(X, \mathbb{Z}) \oplus H^6(X, \mathbb{Z}). 
\end{align*}

The relevant stability condition is self dual w.r.t. the 
derived dual. 
\end{itemize}
In order to define $N_{n, \beta}$, 
 we need to choose an $\mathbb{R}$-ample divisor
 $\omega$, but it can be 
shown that $N_{n, \beta}$ does not 
depend on $\omega$. 
(cf.~Lemma~\ref{Nomega}.)
The self duality 
in defining $L_{n, \beta}$ means that, 
if $E$ is (semi)stable, then its derived 
dual 
\begin{align*}
\dR \hH om(E, \oO_X) \in D^b \Coh(X), 
\end{align*} 
is also (semi)stable. 
The equality $L_{n, \beta}=L_{-n, \beta}$
is a consequence of the self duality. 

In some cases, 
the invariants $N_{n, \beta}$ and $L_{n, \beta}$
are defined in a similar way to DT or PT invariants. 
Let us take $n \in \mathbb{Z}$, $\beta \in H_2(X, \mathbb{Z})$
and an ample $\mathbb{R}$-divisor $\omega$
on $X$. 
Let $M_{n, \beta}(\omega)$ be the moduli space of 
$Z_{\omega}$-semistable objects $E \in \Coh_{\le 1}(X)$
satisfying $\cl_0(E)=(n, \beta)$, in
the notation of Example~\ref{exam:stab} (iii). 
If 
 $n$ and $\beta$ are coprime
and $\omega$ is in a general position
of the ample cone, then 
any $Z_{\omega}$-semistable sheaf $E \in \Coh_{\le 1}(X)$ 
is $Z_{\omega}$-stable, and the moduli space
$M_{n, \beta}(\omega)$ is a projective scheme 
with a symmetric perfect obstruction theory. 
 The invariant $N_{n, \beta}(\omega)$
is defined by 
\begin{align*}
N_{n, \beta}(\omega)  &\cneq \int_{[M_{n, \beta}(\omega)]^{\rm{vir}}} 1 \\
&=\int_{M_{n, \beta}(\omega)}\nu d\chi.
\end{align*}
Here $\nu$ is the Behrend function on 
$M_{n, \beta}(\omega)$. 
We will see in Lemma~\ref{Nomega}
that $N_{n, \beta}(\omega)$
is independent of $\omega$, 
so we can write it as $N_{n, \beta}$. 

On the other hand if $n$ and $\beta$ are not
coprime, then 
$Z_{\omega}$-semistable sheaf may not be $Z_{\omega}$-stable, 
and 
there is no fine moduli space
$M_{n, \beta}(\omega)$
in this case. 
Instead we should work with the moduli stack of 
$Z_{\omega}$-semistable objects, 
denoted by $\mM_{n, \beta}(\omega)$.
The moduli stack $\mM_{n, \beta}(\omega)$
is known to be an Artin stack of finite type over $\mathbb{C}$. 
However 
it is not obvious how to define counting 
invariants via $\mM_{n, \beta}(\omega)$,
since at this moment there is no reasonable 
notion of 
 perfect obstruction theories nor 
virtual fundamental cycles on Artin stacks.
Also it is not obvious how to define the
 weighted Euler characteristic
of $\mM_{n, \beta}(\omega)$, weighted by 
the Behrend function. 
The only known way (at this moment) to do this is 
to introduce the `logarithm' of the moduli stack 
$\mM_{n, \beta}(\omega)$ in the Hall algebra and integrate it. 
We will discuss this 
construction in Section~\ref{sec:Hall}.

As a corollary of Theorem~\ref{main:thm}, we have 
the following result. 
\begin{cor}\emph{\bf{\cite{Tolim2}, \cite{Tcurve1}, \cite{BrH}}}
Conjecture~\ref{conj:GWDT} (i) and Conjecture~\ref{conj:DTPT} are true. 
\end{cor}
\begin{proof}
The property of $N_{n, \beta}$ 
easily implies that the series
\begin{align}\label{rat:series}
\sum_{n>0}
(-1)^{n-1} nN_{n, \beta}q^n, 
\end{align}
is the Laurent expansion of a rational function of 
$q$, invariant under $q\leftrightarrow 1/q$. 
(cf.~\cite[Lemma~4.6]{Tolim2}.)
Then Conjecture~\ref{conj:GWDT} (i) follows from 
the rationality of (\ref{rat:series})
and the property of $L_{n, \beta}$.  

As for Conjecture~\ref{conj:DTPT}, 
the formula (\ref{product2}) in particular implies that 
\begin{align*}
\DT_0(X)=\prod_{n>0}
\exp\left((-1)^{n-1}n N_{n, 0}q^n  \right). 
\end{align*}
Hence the formula (\ref{form:DTPT}) follows. 
\end{proof}

\section{Hall algebras and generalized Donaldson-Thomas invariants}
\label{sec:Hall}
In Subsection~\ref{subsec:product}, we have discussed the invariants
$N_{n, \beta}$ and $L_{n, \beta}$,
which count certain objects the derived category $D^b \Coh(X)$. 
As we discussed there, 
the definition of these invariants is not obvious   
if there is a strictly semistable object. 
In this section, we 
introduce (stack theoretic) 
Hall algebra of coherent sheaves, and 
explain how to construct $N_{n, \beta}$ via 
that algebra.
 The construction is due to Joyce-Song~\cite{JS}, 
which is called \textit{generalized Donaldson-Thomas invariant}. 
(The invariant $L_{n, \beta}$ can be similarly 
constructed, 
and we will discuss it later in Section~\ref{rational}.)

\subsection{Grothendieck groups of varieties}
We recall the notion of Grothendieck groups of 
varieties. 
Let $S$ be a variety over $\mathbb{C}$. 
We define the group $K(\mathrm{Var}/S)$ to be 
the group generated by isomorphism classes 
of symbols 
\begin{align*}
[\rho \colon Y \to S],
\end{align*}
where
$\rho \colon Y \to S$ is an $S$-variety of finite 
type over $\mathbb{C}$, 
and two symbols $[\rho_i \colon Y_i \to S]$ for $i=1, 2$
are isomorphic if there is an isomorphism 
$Y_1 \stackrel{\sim}{\to} Y_2$ preserving the
morphisms $\rho_i$. 
The relation is generated by 
\begin{align*}
[\rho \colon Y \to S] \sim 
[\rho|_{V} \colon V \to S] +[\rho|_{U} \colon 
U \to S],
\end{align*}
where $V\subset Y$ is a closed subvariety 
and $U\cneq Y\setminus V$. 
If $S=\Spec \mathbb{C}$, we 
write $K(\mathrm{Var}/S)$ as $K(\mathrm{Var}/\mathbb{C})$
for simplicity. 

The structure of the group $K(\mathrm{Var}/\mathbb{C})$
is studied in~\cite{Bit}. This is generated by 
smooth projective varieties $[Y]$ with relation given by 
\begin{align}\label{relation}
[\widehat{Y}] -[E] \sim [Y] -[C], 
\end{align}
where $C\subset Y$ is a smooth subvariety, 
$\widehat{Y} \to Y$ is a blow-up at $C$
and $E \subset \widehat{Y}$ is the exceptional divisor. 

Several interesting invariants of varieties 
can be extended to invariants of
elements in $K(\mathrm{Var}/\mathbb{C})$, 
using the above description of the generators and relations.  
For instance for a smooth projective variety $Y$, 
its Poincar\'e polynomial is defined by 
\begin{align}\label{Poincare}
P_t(Y)=\sum_{i=0}^{2 \dim Y}
(-1)^i \dim H^i(Y, \mathbb{C})t^i. 
\end{align}
The polynomial $P_t(\ast)$ 
is compatible with respect to the relation (\ref{relation}), 
hence there is a map, 
\begin{align}\label{Pt}
P_t \colon K(\mathrm{Var}/\mathbb{C}) \to \mathbb{Z}[t], 
\end{align}
such that $P_t([Y])$ coincides with (\ref{Poincare})
if $Y$ is smooth and projective. 

\subsection{Grothendieck groups of stacks}
The notion of Grothendieck group of varieties can be 
generalized to that of Artin stacks. 
For the introduction to stack, the readers 
can consult~\cite{GL}.

Let $\sS$ be 
an Artin stack, locally of finite type over $\mathbb{C}$.  
We define the $\mathbb{Q}$-vector space
 $K(\mathrm{St}/\sS)$ to be generated by isomorphism 
classes of symbols
\begin{align*}
[\rho \colon \yY \to \sS],
\end{align*}
where 
$\yY$ is an Artin stack of finite type over $\mathbb{C}$, 
$\rho$ is a 1-morphism, and two symbols $[\rho_i \colon \yY_i \to \sS]$
for $i=1, 2$
are \textit{isomorphic} if there is a 1-isomorphism of stacks 
$f\colon \yY_1 \stackrel{\sim}{\to} \yY_2$ 
with a 2-isomorphism $\rho_2 \circ f \cong \rho_1$. 
For a technical reason, we assume that $\yY$ 
has affine geometric stabilizers, i.e. 
for any $\mathbb{C}$-valued point $y\in \yY(\mathbb{C})$, 
the automorphism group $\Aut(k(y))$ is an affine algebraic 
group. 
The relation is generated by 
\begin{align*}
[\rho \colon \yY \to \sS] \sim 
[\rho|_{\vV} \colon \vV \to \sS] +
[\rho|_{\uU} \colon \uU \to \sS], 
\end{align*}
where $\vV \subset \yY$ is a closed substack 
and $\uU \cneq \yY \setminus \vV$. 

Let $P_t$ be the map defined in Lemma~\ref{Pt}. 
The following result is proved in~\cite[Theorem~4.10]{Joy5}.
\begin{lem}\label{Pt}
There is a map 
\begin{align*}
P_t \colon K(\mathrm{St}/\sS) \to \mathbb{Q}(t), 
\end{align*}
such that we have 
\begin{align*}
P_t \left(
[\rho \colon [Y/\GL_m(\mathbb{C})] \to \sS] \right)
=\frac{P_t([Y])}{P_t([\GL_m(\mathbb{C})])}.
\end{align*}
Here $Y$ is a quasi-projective variety on which 
$\GL_m(\mathbb{C})$ acts. 
\end{lem} 
\begin{proof}
We sketch an outline of the proof. 
By the assumption that $\yY$ has affine geometric stabilizers, 
we can apply Kresch's result~\cite[Proposition~3.5.9]{Kresch}
to show that any element $u \in K(\mathrm{St}/\sS)$
is written as a finite sum 
\begin{align}\label{exp}
\sum_{i=1}^{k} [\rho_i \colon [Y_i/\GL_{m_i}(\mathbb{C})] \to \sS], 
\end{align}
where $Y_i$ is a quasi-projective variety
on which $\GL_{m_i}(\mathbb{C})$ acts. 
Then we set $P_t(u)$ to be
\begin{align*}
P_t(u)=\sum_{i=1}^{k} \frac{P_t([Y_i])}{P_t([\GL_{m_i}(\mathbb{C})])}.
\end{align*}
The proof given in~\cite[Theorem~4.10]{Joy5}
shows that $P_t(u)$ does not depend on the expression (\ref{exp}). 
\end{proof}
\begin{rmk}
More precisely it is proved in~\cite[Theorem~4.10]{Joy5}
that the map $P_t$ in Lemma~\ref{Pt}
satisfies 
\begin{align*}
P_t \left(
[\rho \colon [Y/G] \to \sS] \right)
=\frac{P_t([Y])}{P_t([G])}.
\end{align*}
Here $Y$ is a quasi-projective variety and 
$G$ is a special algebraic group acting on $Y$,  
where an algebraic group $G$ is called special 
if any principal $G$-bundle is Zariski-localy trivial. 
For instance $\GL_m(\mathbb{C})$, $(\mathbb{C}^{\ast})^k$
are special algebraic groups. 

On the other hand, 
the finite group $\mathbb{Z}/k\mathbb{Z}$ is not special
as $\mathbb{C}^{\ast} \ni z \mapsto z^k \in \mathbb{C}^{\ast}$
is not Zariski locally trivial. 
For instance, let us consider an element 
of the form $[ \rho \colon [\Spec \mathbb{C}/G] \to \sS]$
for $G =\mathbb{Z}/k\mathbb{Z}$.
Then we have 
\begin{align*}
[\Spec \mathbb{C}/G]
\cong [\mathbb{C}^{\ast}/\mathbb{C}^{\ast}],
\end{align*}
 where 
$\mathbb{C}^{\ast}$ acts on $\mathbb{C}^{\ast}$ 
by $g \cdot z=g^k z$. Therefore we have 
\begin{align*}
P_t([[\Spec \mathbb{C}/G] \stackrel{\rho}{\to} \sS])
&=\frac{P_t(\mathbb{C}^{\ast})}{P_t(\mathbb{C}^{\ast})} \\
&=1.
\end{align*}
\end{rmk}
We will need the notions of push-forward and pull-back 
for the groups $K(\mathrm{St}/\sS)$. 
Let $f \colon \sS_1 \to \sS_2$ be a morphism of stacks. 
Then we have the push-forward, 
\begin{align*}
f_{\ast} \colon K(\mathrm{St}/\sS_1) \to K(\mathrm{St}/\sS_2), 
\end{align*}
defined by 
\begin{align*}
f_{\ast}[\rho \colon \yY \to \sS_1]
=[f\circ \rho \colon \yY \to \sS_2].
\end{align*}
Moreover if $f$ is of finite type, then we have the pull-back, 
\begin{align*}
f^{\ast} \colon K(\mathrm{St}/\sS_2) \to K(\mathrm{St}/\sS_1), 
\end{align*}
defined by 
\begin{align*}
f^{\ast}[\rho \colon \yY \to \sS_2]
=[f^{\ast}\rho \colon 
\yY \times_{\sS_2}\sS_1 \to \sS_1].
\end{align*}

\subsection{Hall algebras of coherent sheaves}
For a smooth projective variety $X$ over $\mathbb{C}$, 
we denote by $\mM$ the moduli stack of 
coherent sheaves on $X$. 
Namely $\mM$ is a 2-functor, 
\begin{align}\label{Coh}
\mM \colon (\mathrm{Sch}/\mathbb{C}) \to 
(\mathrm{groupoid}), 
\end{align}
which sends a $\mathbb{C}$-scheme $S$ to the 
groupoid whose objects consist of
flat families of coherent sheaves over $S$, 
\begin{align*}
\eE \in \Coh(X\times S).
\end{align*}
It is well-known that $\mM$ is an 
Artin stack which is locally of finite 
type over $\mathbb{C}$. 
\begin{defi}
We define the $\mathbb{Q}$-vector space $H(X)$ 
to be 
\begin{align*}
H(X) \cneq K(\mathrm{St}/\mM). 
\end{align*}
\end{defi}
We introduce the $\ast$-product on the 
$\mathbb{Q}$-vector space $H(X)$. 
Let $\eE x$ be the 2-functor, 
\begin{align*}
\eE x \colon (\mathrm{Sch}/\mathbb{C}) \to (\mathrm{groupoid}), 
\end{align*}
which sends a $\mathbb{C}$-scheme $S$ to the 
groupoid whose objects consist of exact sequences
in $\Coh (X \times S)$, 
\begin{align}\label{ex:seq}
0 \to \eE_1 \to \eE_2 \to \eE_3 \to 0, 
\end{align}
such that each $\eE_i$ is flat over $S$. 
The stack $\eE x$ is also an Artin stack 
locally of finite type over $\mathbb{C}$. 
There are 1-morphisms, 
\begin{align*}
p_i \colon \eE x \to \mM, \ i=1, 2, 3, 
\end{align*}
which send an exact sequence (\ref{ex:seq})
to the object $\eE_i$. In particular we have the diagram, 
\begin{align*}
\xymatrix{
\eE x \ar[r]^{p_2} \ar[d]_{(p_1, p_3)} & \mM, \\
\mM \times \mM. &  
}
\end{align*}
Also we define the map 
\begin{align*}
\iota \colon 
H(X) \otimes H(X) \to K(\mathrm{St}/\mM \times \mM),
\end{align*}
as follows: 
\begin{align*}
\iota([\yY_1 \stackrel{\rho_1}{\to} \mM]\otimes 
[\yY_1 \stackrel{\rho_1}{\to} \mM])
=[\yY_1 \times \yY_2 \stackrel{\rho_1 \times \rho_2}{\to} \mM \times \mM].
\end{align*}
We define $\ast$-product on $H(X)$ to be
\begin{align}\label{product}
\ast =p_{2\ast}(p_1, p_3)^{\ast} \iota \colon 
H(X) \otimes H(X) \to H(X). 
\end{align}
The following result is proved in~\cite{Joy2}. 
\begin{thm}\emph{\bf{\cite[Theorem~5.2]{Joy2}}}
$(H(X), \ast)$ is an associative algebra with unit 
given by $\delta_0=[\Spec \mathbb{C} \stackrel{\rho}{\to} \mM]$.
Here $\rho(\cdot)=0 \in \Coh(X)$.  
\end{thm}
Let us look at the $\ast$-product for 
`delta-functions',
corresponding to objects $E_1, E_2 \in \Coh(X)$. Namely 
for an object $E\in \Coh(X)$, we set
\begin{align*}
\delta_{E}=[\rho_E \colon \Spec \mathbb{C} \to \mM], \quad 
\rho_E(\cdot)=E. 
\end{align*}
The $\ast$-product $\delta_{E_1} \ast \delta_{E_2}$
can be written as 
\begin{align}\label{ast}
\delta_{E_1} \ast \delta_{E_2}=
\left[\rho \colon \left[\frac{\Ext^{1}(E_2, E_1)}
{\Hom(E_2, E_1)}\right] \to \mM \right]. 
\end{align}
Here $\rho$ is a map sending an element 
$u\in \Ext^1(E_2, E_1)$ to the object $E_3 \in \Coh(X)$, 
which fits into the exact sequence, 
\begin{align}\label{E123}
0\to E_1 \to E_3 \to E_2 \to 0, 
\end{align}
with extension class $u$. The vector space $\Hom(E_2, E_1)$
acts on $\Ext^1(E_2, E_1)$ trivially. 
In fact the $\mathbb{C}$-valued points of
the fiber product
\begin{align}\label{fibre:prod}
(\mM\times \mM)\times_{(\rho_{E_1} \times \rho_{E_2})} \Spec \mathbb{C},
\end{align}
bijectively correspond to the exact sequences (\ref{E123}), 
hence elements in $\Ext^1(E_2, E_1)$. 
Given such an extension, the group of the automorphisms
of the stack (\ref{fibre:prod})
at the $\mathbb{C}$-valued point (\ref{E123}) is 
the kernel of the natural map, 
\begin{align*}
\Aut(0 \to E_1 \to E_3 \to E_2 \to 0) \to \Aut(E_1) \times \Aut(E_2), 
\end{align*}
which is isomorphic to $\Hom(E_2, E_1)$. 
Hence we have the description (\ref{ast}).

\subsection{Semistable one or zero dimensional sheaves}
In this subsection, we assume that $X$ is a smooth 
projective Calabi-Yau 3-fold over $\mathbb{C}$. 
Let $\omega$ be an $\mathbb{R}$-ample divisor 
on $X$. Recall that we constructed a
stability condition $Z_{\omega}$ on the subcategory 
\begin{align*}
\Coh_{\le 1}(X) \subset \Coh(X),
\end{align*}
in Example~\ref{exam:stab}.
Given an element $(n, \beta) \in \mathbb{Z} \oplus H_2(X, \mathbb{Z})$, 
we have the substack, 
\begin{align}\label{substack}
\mM_{n, \beta}(\omega) \subset \mM, 
\end{align}
which parameterizes  $Z_{\omega}$-semistable $E\in \Coh_{\le 1}(X)$
satisfying 
\begin{align}\label{chi=bn}
(\chi(E), [E])=(n, \beta).
\end{align}
The substack (\ref{substack})
is known to be an open substack of $\mM$, which is of finite type 
over $\mathbb{C}$. 
Furthermore suppose that
 $\beta$ and $n$ are coprime and $\omega$
is in a general position in the ample cone.
Then any $Z_{\omega}$-semistable 
object $E\in \Coh_{\le 1}(X)$
satisfying (\ref{chi=bn})
is $Z_{\omega}$-stable, and the stack
 $\mM_{n, \beta}(\omega)$
is a $\mathbb{C}^{\ast}$-gerbe over 
a projective scheme $M_{n, \beta}(\omega)$, i.e. 
\begin{align}\label{roshin}
\mM_{n, \beta}(\omega) \cong [M_{n, \beta}(\omega)/\mathbb{C}^{\ast}].
\end{align} 
Here $\mathbb{C}^{\ast}$-acts on $M_{n, \beta}(\omega)$
trivially. 
The substack (\ref{substack}) 
defines the element of $H(X)$, 
\begin{align*}
\delta_{n, \beta}(\omega) =[\mM_{n, \beta}(\omega) 
\hookrightarrow \mM] \in H(X). 
\end{align*}
Recall that we constructed a map,
\begin{align*}
P_t \colon H(X) \to \mathbb{Q}(t), 
\end{align*}
in Lemma~\ref{Pt}. 
Applying $P_t$ to $\delta_{n, \beta}(\omega)$, 
we obtain the element
\begin{align*}
P_t(\delta_{n, \beta}(\omega)) \in \mathbb{Q}(t),
\end{align*} 
which is interpreted as a `Poincar\'e polynomial'
of the moduli stack $\mM_{n, \beta}(\omega)$. 

Suppose that
 $\mM_{n, \beta}(\omega)$ is 
written as (\ref{roshin}).
Then 
we have 
\begin{align}\label{substitute}
(t^2 -1)P_t(\delta_{n, \beta}(\omega)) &=
P_t(\mathbb{C}^{\ast})P_t(\delta_{n, \beta}(\omega)) \\
\notag
&= P_t(M_{n, \beta}(\omega)).
\end{align}
Hence 
we can substitute $t=1$ 
to (\ref{substitute}) 
and obtain,
\begin{align}\label{limt}
\lim_{t\to 1}
(t^2 -1)P_t(\delta_{n, \beta}(\omega))=\chi(M_{n, \beta}(\omega)).
\end{align}
 However when $n$ and $\beta$ are not coprime, 
then $\mM_{n, \beta}(\omega)$ is not 
necessary written as (\ref{roshin}).
In this case, as the following example 
indicates, 
the rational function (\ref{substitute})
 may have 
a pole at $t=1$
so the limit (\ref{limt}) does not make sense. 
\begin{exam}\label{exam:doesnot}
Let $C \cong \mathbb{P}^1 \subset X$ be a 
super rigid rational curve
as in Example~\ref{GW:rigid}. 
Then we have 
\begin{align*}
\mM_{0, k[C]}(\omega)
\cong [\Spec \mathbb{C} /\GL_k(\mathbb{C})], 
\end{align*}
whose closed point correspond to $\oO_C(-1)^{\oplus k}$.
Therefore using~\cite[Lemma~4.6]{Joy5}, we have 
\begin{align*}
(t^2 -1)P_t(\delta_{0, k[C]}(\omega))
&=(t^2 -1) \frac{1}{P_t(\GL_k(\mathbb{C}))}  \\
&=\frac{t^{k^2 -k}}{t^2(t^4-1) \cdots (t^{2k}-1)},
\end{align*}
and the limit $t\to 1$ does not exist when $k \ge 2$. 
\end{exam}
 
Instead, we take the `logarithm' of 
$\delta_{n, \beta}(\omega)$ in $H(X)$. 
\begin{defi}
We define $\epsilon_{n, \beta}(\omega) \in H(X)$ 
to be 
\begin{align}\label{log}
\epsilon_{n, \beta}(\omega)
=
\sum_{\begin{subarray}{c}
l\ge 1, n_i \in \mathbb{Z}, \beta_i \in H_2(X, \mathbb{Z}), 
1\le i\le l, \\
n_1 +\cdots +n_l=n, 
\beta_1 +\cdots +\beta_l =\beta, \\
\arg Z_{\omega}(n_i, \beta_i)=\arg Z_{\omega}(n, \beta).
\end{subarray}}
\frac{(-1)^{l-1}}{l}
\delta_{n_1, \beta_1}(\omega) \ast 
\cdots \ast \delta_{n_l, \beta_l}(\omega). 
\end{align}
\end{defi}
Namely
for each ray $l \subset \mathbb{H}$,
if we set 
\begin{align*}
\delta_{l}(\omega) &=1+
\sum_{Z_{\omega}(n, \beta) \in l}
\delta_{n, \beta}(\omega), \\
\epsilon_{l}(\omega) &=
\sum_{Z_{\omega}(n, \beta)\in l}
\epsilon_{n, \beta}(\omega), 
\end{align*}
then 
we have 
\begin{align*}
\epsilon_{l}(\omega)
=\log \delta_{l}(\omega). 
\end{align*}
It is shown in~\cite[Section~6.2]{Joy5} that the
function $(t^2-1)P_t(\epsilon_{n, \beta}(\omega))$
has the limit $t\to 1$, hence 
we obtain the invariant, 
\begin{align*}
\widehat{N}_{n, \beta}(\omega)=
\lim_{t\to 1}(t^2-1)P_t(\epsilon_{n, \beta}(\omega)) \in \mathbb{Q}. 
\end{align*} 
The invariant $\widehat{N}_{n, \beta}(\omega)$
is interpreted as an `Euler characteristic'
of the moduli stack $\mM_{n, \beta}(\omega)$.

\subsection{Invariants $N_{n, \beta}$}
The invariant $\widehat{N}_{n, \beta}(\omega)$
is interpreted as an unweighted Euler characteristic
of $\mM_{n, \beta}(\omega)$, 
and we need to 
involve the 
Behrend function in 
order to construct DT type invariants. 
It is easy to 
extend the notion of the Behrend function to 
the locally constructible function on 
the Artin stack $\mM$, 
\begin{align*}
\nu_{\mM} \colon \mM \to \mathbb{Z}, 
\end{align*}
so that if $M \to \mM$ is any atlas
of relative dimension $d$, 
then $\nu_{\mM}=(-1)^{d} \nu_{M}$. 
(cf.~\cite[Proposition~4.4]{JS}.)
We define the map 
\begin{align}\label{map:nu}
\nu \cdot \colon H(X) \to H(X), 
\end{align}
by sending an 
element $[\rho \colon \yY \to \mM]$
to the element, 
\begin{align*}
\sum_{i \in \mathbb{Z}}i
[\rho|_{\yY_i} \colon \yY_i \to \mM], 
\end{align*}
where $\yY_i =(\nu_{\mM} \circ \rho)^{-1}(i)$. 
\begin{defi}
We define $N_{n, \beta}(\omega)$ to be
\begin{align*}
N_{n, \beta}(\omega)=
\lim_{t\to 1}(t^2-1)P_t(-\nu\cdot \epsilon_{n, \beta}(\omega)) \in \mathbb{Q}. 
\end{align*} 
\end{defi}
Again the existence of the limit $t\to 1$
is proved in~\cite[Section~6.2]{Joy5}. 
A priori, the invariant $N_{n, \beta}(\omega)$
is defined after we choose a
polarization $\omega$. However
we have the following: 
\begin{lem}\label{Nomega}
The invariant $N_{n, \beta}(\omega)$
does not depend on a choice of $\omega$. 
\end{lem}
\begin{proof}
The result is proved in~\cite[Theorem~6.16]{JS}.
\end{proof}
In what follows, we set
\begin{align*}
N_{n, \beta} \cneq N_{n, \beta}(\omega), 
\end{align*}
for some ample divisor $\omega$ on $X$.

\begin{exam}
(i) 
Suppose that 
$n$ and $\beta$ are coprime, and 
$\omega$ is in a general position. 
Then 
$\mM_{n, \beta}(\omega)$ is 
written as (\ref{roshin}) for 
a projective scheme $M_{n, \beta}(\omega)$.
Let $\nu_{M}$ be the Behrend
 function on $M_{n, \beta}(\omega)$. 
Then we have 
\begin{align*}
\epsilon_{n, \beta}(\omega)=\delta_{n, \beta}(\omega), \
\nu_{\mM}|_{\mM_{n, \beta}(\omega)}=-\nu_{M},
\end{align*} hence 
we have 
\begin{align*}
N_{n, \beta}
&=\int_{M_{n, \beta}(\omega)}\nu_{M} d\chi, \\
&= \int_{[M_{n, \beta}(\omega)]^{\rm vir}}1. 
\end{align*}

(ii)
In the situation of Example~\ref{exam:doesnot},
we have 
\begin{align*}
\delta_{0, [C]}(\omega)
 =\left[\frac{\Spec \mathbb{C}}{\mathbb{C}^{\ast}}\right], \ 
\delta_{0, 2[C]}(\omega)
=\left[\frac{\Spec \mathbb{C}}{\GL_{2}(\mathbb{C})}\right].
\end{align*}
Therefore we have 
\begin{align*}
\epsilon_{0, 2[C]}(\omega)
&= \delta_{0, 2[C]}(\omega) -\frac{1}{2} \delta_{0, [C]}(\omega)
\ast \delta_{0, [C]}(\omega) \\
&=\left[\frac{\Spec \mathbb{C}}{\GL_2(\mathbb{C})}
\to \mM \right] -\frac{1}{2
}\left[\frac{\Spec \mathbb{C}}{\mathbb{C}^{\ast}}
\to \mM \right] 
\ast \left[\frac{\Spec \mathbb{C}}{\mathbb{C}^{\ast}}
\to \mM \right], \\
&=\left[\frac{\Spec \mathbb{C}}{\GL_2(\mathbb{C})}
\to \mM \right] -\frac{1}{2
}\left[\frac{\Spec \mathbb{C}}{\mathbb{A}^1 \rtimes
(\mathbb{C}^{\ast})^2}
\to \mM \right].
\end{align*} 
The Behrend function $\nu_{\mM}$ is $1$
on $\oO_C(-1)^{\oplus 2}$, hence we have 
\begin{align*}
& (t^2 -1)P_t(-\nu \cdot \epsilon_{0, 2[C]}(\omega)) \\
&= (t^2-1)\left\{-\frac{1}{t^2(t^2-1)(t^4-1)}
+\frac{1}{2 t^2(t^2-1)^2} \right\} \\
&=\frac{1}{2t^2(t^2 +1)}. 
\end{align*}
By taking the limit $t\to 1$, we obtain 
$N_{0, 2[C]}=1/4$. In  general, it can be proved that 
(cf.~\cite[Example~6.2]{JS})
\begin{align*}
N_{0, k[C]}=\frac{1}{k^2}. 
\end{align*}

(iii) Let us consider 
the case $\beta=0$. In this case, 
$\mM_{n, 0}(\omega)$ is a moduli 
stack of length $n$ zero dimensional sheaves. 
Explicitly $\mM_{n, 0}(\omega)$
is described as follows. 
Let $\mathrm{Quot}^{(n)}(\oO_X^{\oplus n})$
be the Grothendieck Quot scheme 
which parameterizes quotients
 \begin{align}\label{quot:s}
\oO_X ^{\oplus n} \twoheadrightarrow F,
\end{align}
with $F$ zero dimensional length $n$ sheaves. 
The group $\GL_n(\mathbb{C})$ acts on 
$\mathrm{Quot}^{(n)}(\oO_X^{\oplus n})$
via 
\begin{align*}
g \cdot (\oO_X^{\oplus n} \stackrel{s}{\twoheadrightarrow} F)  
=(\oO_X^{\oplus n} \stackrel{s\circ g}{\twoheadrightarrow} F), \quad
g\in \GL_n(\mathbb{C}). 
\end{align*} 
Let 
\begin{align*}
U^{(n)}\subset \mathrm{Quot}^{(n)}(\oO_X^{\oplus n}),
\end{align*}
be the open subscheme corresponding to 
quotients (\ref{quot:s})
such that the induced morphism 
$H^0(s) \colon \mathbb{C}^{\oplus n} \to H^0(F)$ is an isomorphism. 
The $\GL_n(\mathbb{C})$-action on $\mathrm{Quot}^{(n)}(\oO_X^{\oplus n})$
preserves $U^{(n)}$, and 
the moduli stack $\mM_{n, 0}(\omega)$ is written as 
\begin{align*}
\mM_{n, 0}(\omega) \cong [U^{(n)}/\GL_n(\mathbb{C})]. 
\end{align*}
In principle, it may be possible to 
calculate $N_{n, 0}$ using the above 
description of the moduli stack. 
(For instance, the computation in~\cite[Section~5]{Trk2}
is applied for $n=2$.)
However at this moment, 
a computation of $N_{n, 0}$
for $n\ge 3$
is not yet done along with this argument. 
Instead, we can compute $N_{n, 0}$
using the wall-crossing formula 
and the 
computation of $\DT_0(X)$
in Example~\ref{ex:DT} (i). 
The result is given in~\cite[Paragraph~6.3]{JS}, 
\cite[Paragraph~6.4]{K-S}, \cite[Remark~5.14]{Tcurve1}, 
\begin{align}\label{D0}
N_{n, 0}=-\chi(X) \sum_{k|n, \ k\ge 1} \frac{1}{k^2}
\end{align}
\end{exam}

\section{Wall-crossing in D0-D2-D6 bound states}\label{rational}
Let $X$ be a smooth projective Calabi-Yau 3-fold over $\mathbb{C}$. 
In this section, we explain how to deduce the 
product formula (\ref{product1})
by using the wall-crossing formula. 
In principle, the result is obtained by combining 
the arguments in~\cite{Tolim2}, 
Joyce-Song's wall-crossing formula~\cite{JS} 
and the announced result by Behrend-Getzler~\cite{BG}.  
However the arguments in~\cite{Tolim2}
are complicated, and we simplify 
the arguments by using the 
framework of~\cite{Tcurve1}.
\subsection{Category of D0-D2-D6 bound states}
We define the category $\aA_X$ as follows: 
\begin{align*}
\aA_X \cneq \langle \oO_X, \Coh_{\le 1}(X)[-1] \rangle_{\ex}. 
\end{align*}
In~\cite[Lemma~3.5]{Tcurve1}, it is proved that 
$\aA_X$ is the heart of a bounded t-structure
on $\dD_X$, 
\begin{align*}
\dD_X =\langle \oO_X, \Coh_{\le 1}(X) \rangle_{\tr} \subset
D^b \Coh(X),
\end{align*}
hence in particular $\aA_X$ is 
an abelian category. 
The triangulated category 
$\dD_X$ is called the 
category of \textit{D0-D2-D6 bound states}. 

 The heart $\aA_X$ has
properties which are 
required in discussing DT/PT correspondence
in Subsection~\ref{subsec:PT}. 
 For instance 
if we consider an ideal sheaf $I_Z$ for 
a subscheme $Z\subset X$ with $\dim Z \le 1$, we have 
the distinguished triangle, 
\begin{align}\label{OIO}
\oO_Z[-1] \to I_Z \to \oO_X.
\end{align}
Since $\oO_Z[-1]$ and $\oO_X$ 
are objects in $\aA_X$, it follows that 
$I_Z \in \aA_X$ and the sequence (\ref{OIO})
is an exact sequence in $\aA_X$. 
Also for a stable pair $(F, s)$, 
let $I^{\bullet}=(\oO_X \stackrel{s}{\to} F)$
be the associated two term complex 
with $\oO_X$ located in degree zero and 
$F$ in degree one. Then $I^{\bullet}$ fits into 
the distinguished triangle, 
\begin{align}\label{FIO}
F[-1] \to I^{\bullet} \to \oO_X.
\end{align}
By the same argument as above, 
we have $I^{\bullet}\in \aA_X$
and the sequence (\ref{FIO}) is an exact sequence in $\aA_X$. 
As the above argument indicates, 
the heart $\aA_X$ is 
expected to be an important category 
in studying curve counting invariants on Calabi-Yau 
3-folds. 
\subsection{Comparison with perverse coherent sheaves}
In~\cite{Bay}, \cite{Tolim},
the notions of \textit{polynomial stability}
and \textit{limit stability} are 
introduced on the following category of 
\textit{perverse coherent sheaves},
\begin{align*}
\aA^p \cneq \langle \Coh_{\ge 2}(X)[1], \Coh_{\le 1}(X) \rangle_{\ex}.
\end{align*}
Here $\Coh_{\ge 2}(X)$ is the right orthogonal complement of $\Coh_{\le 1}(X)$
in $\Coh(X)$. 
In this subsection, 
we compare $\aA_X$ with $\aA^p$. 

Obviously we have 
\begin{align*}
\aA_X \subset \aA^p[-1]. 
\end{align*}
By~\cite[Lemma~2.16]{Tolim}, 
there exists a torsion pair $(\aA_{1}^p, \aA_{1/2}^p)$
on $\aA^p$, defined by 
\begin{align*}
\aA_{1}^p &\cneq \langle F[1], \oO_x : 
F \mbox{ is pure two dimensional, }
x\in X \rangle_{\ex}, \\
\aA_{1/2}^p &\cneq \{ E\in \aA^p : \Hom(F, E)=0 \mbox{ for any }
F\in \aA_{1}^p \}. 
\end{align*}
Namely we have the following, (cf.~\cite{HRS},)
\begin{itemize}
\item For any $T \in \aA_{1}^p$
and $F\in \aA_{1/2}^p$, we have $\Hom(T, F)=0$. 
\item For any $E\in \aA^p$, there is an exact 
sequence 
\begin{align}\notag
0 \to T \to E \to F \to 0,
\end{align}
with $T\in \aA_{1}^p$ and $F\in \aA_{1/2}^p$. 
\end{itemize}
We set 
\begin{align}\label{AX1}
\aA_{X, 1} &\cneq \aA_{1}^p[-1] \cap \aA_X, \\
\notag
&= \langle \oO_x[-1] : x \in X \rangle_{\ex},
\end{align}
and 
\begin{align}\label{AX2}
\aA_{X, 1/2} &\cneq \aA_{1/2}^p [-1] \cap \aA_X \\
\notag
&=\{ E \in \aA_X : \Hom(\aA_{X, 1}, E)=0\}.
\end{align} 
It is easy to check that $(\aA_{X, 1}, \aA_{X, 1/2})$
is a torsion pair on $\aA_X$, 
using the fact that $\aA_X$ is noetherian. 
(cf.~\cite[Lemma~6.2]{Tcurve1}.)
We have the following lemma. 
\begin{lem}\label{lem:rank:c1}
For an object $E \in \aA_{1/2}^p[-1]$, suppose that 
\begin{align*}
\rank(E) \in \{0, 1\}, \ c_1(E)=0. 
\end{align*}
Then we have $E \in \aA_{X, 1/2}$.
\end{lem}
\begin{proof}
We only prove the case of $\rank(E)=1$. 
Take $E \in \aA_{1/2}^p[-1]$ with 
$\rank(E)=1$ and $c_1(E)=0$. Then 
by~\cite[Lemma~3.2]{Tolim}, 
we have the exact sequence in $\aA^p[-1]$, 
\begin{align*}
I_C \to E \to F[-1],
\end{align*}
for some curve $C\subset X$ and $F \in \Coh_{\le 1}(X)$. 
Since $I_C, F[-1] \in \aA_X$, we have 
$E \in \aA_X$, hence $E \in \aA_{X, 1/2}$. 
\end{proof}
Below we use the following notation. 
For $E, F \in \aA_{1/2}^p$, a morphism 
$u \colon E \to F$ in $\aA^p$ is 
called a \textit{strict monomorphism} if 
$u$ is injective in $\aA^p$ and
$\Cok(u) \in \aA_{1/2}^p$. Similarly 
$u$ is called a \textit{strict epimorphism}
if $u$ is surjective in $\aA^p$ 
and $\Ker(u) \in \aA_{1/2}^p$. 
By replacing $(\aA_{i}^p, \aA^p)$
by $(\aA_{X, i}, \aA_X)$, we have 
the notions of strict monomorphism, 
strict epimorphism on $\aA_{X, i}$.

\subsection{Weak stability conditions on $\aA_X$}
In this subsection, we construct weak stability conditions on 
$\aA_X$. (cf.~Definition~\ref{subsec:Weak}.)
The finitely generated free abelian 
group $\Gamma$ is defined by 
\begin{align*}
\Gamma &\cneq \mathbb{Z} \oplus H_2(X, \mathbb{Z})
\oplus \mathbb{Z}, \\
 &=\Gamma_0 \oplus \mathbb{Z},
\end{align*} 
where $\Gamma_0$ is introduced in 
Example~\ref{exam:stab} (iii). 
Below we write an element in $\Gamma$
as $(n, \beta, r)$ for $n\in \mathbb{Z}$, 
$\beta \in H_2(X, \mathbb{Z})$
and $r\in \mathbb{Z}$.
For an object $E\in \aA_X$, note that  
 \begin{align}\label{chH2i}
\ch_i(E) \in H^{2i}(X, \mathbb{Z}), 
\end{align}
since (\ref{chH2i})
is true for the generating 
set of objects $\oO_X$ and $E \in \Coh_{\le 1}(X)[-1]$. 
Therefore the Chern characters define the group homomorphism, 
\begin{align*}
\cl \colon K(\aA_X) \to \Gamma, 
\end{align*}
given by 
\begin{align*}
\cl(E) =(\ch_3(E), \ch_2(E), \ch_0(E)).
\end{align*}
Here we have identified 
$H^0(X, \mathbb{Z})$ and $H^6(X, \mathbb{Z})$
with $\mathbb{Z}$, and $H^2(X, \mathbb{Z})$
with 
$H_2(X, \mathbb{Z})$
via Poincar\'e duality. 
We take the following 
2-step filtration in $\Gamma$, 
\begin{align*}
0=\Gamma_{-1}\subsetneq \Gamma_0 \subsetneq \Gamma_1 =\Gamma, 
\end{align*}
where $\Gamma_0$ is
given in Example~\ref{exam:stab}, 
and the embedding 
$\Gamma_0 \hookrightarrow \Gamma$
 is given by $(n, \beta) \mapsto (n, \beta, 0)$.  
Hence each subquotient is given by 
\begin{align*}
\Gamma_{0}/\Gamma_{-1} &= \mathbb{Z}
\oplus H_2(X, \mathbb{Z}), \\
\Gamma_{1}/\Gamma_{0} &=\mathbb{Z}. 
\end{align*}
Given a following data, 
\begin{align}\label{dot}
\omega \in H^2(X, \mathbb{Q}), \quad 
0<\theta <1, 
\end{align}
where $\omega$ is an ample class,
we construct
\begin{align}\label{Zoti}
Z_{\omega, \theta}=\{Z_{\omega, \theta, i}\}_{i=0}^{1} 
\in \prod_{i=0}^{1} \Hom(\Gamma_i/\Gamma_{i-1}, \mathbb{C}), 
\end{align}
as follows:
\begin{align*}
Z_{\omega, \theta, 0}(n, \beta) &=
n-(\omega \cdot \beta)\sqrt{-1},  
 \\
Z_{\omega, \theta, 1}(r) &=r \exp(i\pi \theta).
\end{align*}
Here $(n, \beta) \in \mathbb{Z} \oplus H_2(X, \mathbb{Z})$
and $r\in \mathbb{Z}$. 
We have the following lemma. 
\begin{lem}\label{lem:system}
The system of group homomorphisms (\ref{Zoti}) is a weak 
stability condition on $\aA_X$. 
\end{lem}
\begin{proof}
For an object $E\in \aA_X$, let 
us take $i\in \{0, 1\}$ 
so that $\cl(E) \in \Gamma_{i} \setminus \Gamma_{i-1}$
is satisfied.
If $i=1$, then 
\begin{align*}
Z_{\omega, \theta}(E) \in \mathbb{R}_{>0} 
\exp (i\pi \theta) \subset \mathbb{H}.
\end{align*}
Also if $i=0$, then $E \in \Coh_{\le 1}(X)[-1]$ and 
\begin{align*}
Z_{\omega, \theta}(E)=Z_{\omega}(E[1]) \in \mathbb{H}, 
\end{align*}
where $Z_{\omega}$ is 
defined in Example~\ref{exam:stab} (iii). 
Therefore the condition (i) in Definition~\ref{def:wstab}
is satisfied. 

We check the condition (ii) in Definition~\ref{def:wstab}. 
Let  
$(\aA_{X, 1}, \aA_{X, 1/2})$ be the torsion 
pair of $\aA_X$, given by (\ref{AX1}), (\ref{AX2}). 
For any $E\in \aA_X$, there is an exact sequence 
in $\aA_X$, 
\begin{align}\label{TF2}
0 \to T \to E \to F \to 0,
\end{align}
with $T \in \aA_{X, 1}$ and $F \in \aA_{X, 1/2}$. 
By~\cite[Lemma~2.19]{Tolim}, 
the categories $\aA_{X, 1}$ and $\aA_{X, 1/2}$
are
finite length (i.e. noetherian and 
artinian with respect to strict epimorphism and 
strict monomorphism)
quasi-abelian categories. 
(See~\cite[Section~4]{Brs1} for the definition
 of 
quasi-abelian categories.)

On the other hand, 
by the same argument of~\cite[Lemma~2.27]{Tolim},
 an object $E\in \aA_X$ is 
$Z_{\omega, \theta}$-semistable if 
and only if 
one of the following conditions holds: 
\begin{itemize}
\item We have $E\in \aA_{X, 1}$. 
\item We have $E\in \aA_{X, 1/2}$, and for any 
exact sequence
\begin{align*}
0 \to A \to E \to B \to 0
\end{align*}
in $\aA_X$
with $A, B \in \aA_{X, 1/2}$, we have 
\begin{align}\label{ZAB}
\arg Z_{\omega, \theta}(A) \le \arg Z_{\omega, \theta}(B).
\end{align} 
\end{itemize}
Then for any $E \in \aA_{X}$, 
its Harder-Narasimhan filtration 
is obtained 
by combining 
the sequence (\ref{TF2}) and the 
Harder-Narasimhan filtration of $F$, 
where $F$ is given by the sequence 
(\ref{TF2}). 
The existence of the latter filtration 
is 
ensured by the fact that 
$Z_{\omega, \theta}$-semistable objects in 
$\aA_{X, 1/2}$ are 
characterized by the inequality (\ref{ZAB})
for exact sequences in $\aA_{X, 1/2}$, and 
$\aA_{X, 1/2}$ is of 
finite length.
(See the proof of~\cite[Theorem~2.29]{Tolim}.)
\end{proof}
We remark that the abelian category $\aA_X$ contains 
the subcategory, 
\begin{align*}
\Coh_{\le 1}(X)[-1] \subset \aA_X, 
\end{align*}
which is closed under subobjects and quotients. 
Hence for $F\in \Coh_{\le 1}(X)$, 
the object $F[-1] \in \aA_X$ is 
$Z_{\omega, \theta}$-(semi)stable if and only 
if $F$ is $Z_{\omega}$-(semi)stable in the sense of 
Example~\ref{exam:stab} (iii). 

Let 
\begin{align*}
\Stab_{\Gamma_{\bullet}}(\dD_X)
\end{align*}
be the space of weak stability conditions on 
$\dD_X$, as in (\ref{StabG}).
It is straightforward to check that the 
pairs $(Z_{\omega, \theta}, \aA_X)$
satisfy the conditions
required to construct the space 
$\Stab_{\Gamma_{\bullet}}(\dD_X)$, 
i.e. local finiteness, support property
in~\cite[Section~2]{Tcurve1}.
Therefore by applying~\cite[Lemma~7.1]{Tcurve1}, we have 
the continuous morphism for a fixed $\omega$, 
\begin{align}\label{cmap}
(0, 1) \ni \theta \mapsto 
(Z_{\omega, \theta}, \aA_X) \in \Stab_{\Gamma_{\bullet}}(\dD_X). 
\end{align} 

\subsection{Comparison with $\mu$-limit stability}
Let us take 
\begin{align*}
B+i\omega \in H^2(X, \mathbb{C})
\end{align*}
with $\omega$ ample.
Below we set $B=k\omega$ for $k\in \mathbb{R}$. 
 In~\cite{Tolim2},
the author introduced the notion of 
$\mu_{B+i\omega}$-\textit{limit stability} 
on the abelian category $\aA^p$.  
Suppose that an object $E\in \aA^p[-1]$ satisfies 
\begin{align}\label{chE}
\ch(E)=(1, 0, -\beta, -n) \in H^0 \oplus H^2 \oplus H^4 \oplus H^6,
\end{align}
Then by~\cite[Lemma~3.8]{Tolim2} and~\cite[Proposition~3.13]{Tolim2},
an object $E[1] \in \aA^p$
 is $\mu_{B+i\omega}$-limit
semistable if and only if
$E \in \aA_{1/2}^p$ and 
the following conditions are satisfied:
\begin{itemize}
\item For any pure one dimensional sheaf $0\neq F$
which admits a strict monomorphism $F \hookrightarrow E[1]$
in $\aA_{1/2}^p$, we have 
$\ch_3(F)/\omega \ch_2(F) \le -2k$.
\item For any pure one dimensional sheaf $0\neq G$
which admits a strict epimorphism $E[1] \twoheadrightarrow G$
in $\aA_{1/2}^p$, we have 
$\ch_3(G)/\omega \ch_2(G) \ge -2k$.
\end{itemize}
Now we set
\begin{align}\label{2theta}
k=\frac{1}{2\tan \pi \theta}. 
\end{align}
Here $k=0$ if $\theta=1/2$. 
By Lemma~\ref{lem:rank:c1} and the arguments in
the proof of Lemma~\ref{lem:system},
the following lemma obviously follows. 
\begin{lem}\label{lem:compare}
Take $k$ and $\theta$ satisfying (\ref{2theta}). 
Then for an object $E \in \aA^p[-1]$ satisfying (\ref{chE}), 
$E[1] \in \aA^p$ 
is $\mu_{k\omega +i\omega}$-limit 
semistable in the sense of~\cite[Section~3]{Tolim2}
if and only if 
$E \in \aA_X$ and $E$ is $Z_{\omega, \theta}$-semistable
satisfying 
\begin{align*}
\cl(E)=(-n, -\beta, 1) \in \Gamma. 
\end{align*} 
\end{lem}

\subsection{Moduli stacks of semistable objects}
In this subsection, we discuss moduli stacks of 
semistable objects in $\aA_X$. 
We denote by $\widehat{\mM}$ the 2-functor, 
\begin{align*}
\widehat{\mM} \colon (\mathrm{Sch}/\mathbb{C}) 
\to (\mathrm{groupoid}), 
\end{align*} 
which sends a $\mathbb{C}$-scheme $S$ 
to the groupoid whose objects consist of 
objects 
\begin{align*}
\eE \in D(\Coh(X\times S)),
\end{align*}
such that  
\begin{itemize}
\item The object $\eE$ is relatively perfect over $S$. 
(See~\cite[Definition~2.1.1]{LIE}.)
In particular for each $s\in S$, 
we have the derived pull-back, 
\begin{align}\label{der:back}
\eE_{s} \cneq \dL i_{s}^{\ast} \eE \in D^b \Coh(X).
\end{align} 
Here $i_s \colon X \times \{s \} \hookrightarrow X \times S$
is the inclusion. 
\item The object (\ref{der:back}) 
satisfies 
\begin{align*}
\Ext^{i}(\eE_s, \eE_s)=0, \quad i<0, 
\end{align*}
for any $s\in S$. 
\end{itemize}
By the result of Lieblich~\cite{LIE}, 
the 2-functor $\widehat{\mM}$ 
is an Artin stack locally of finite type over $\mathbb{C}$. 
We note that the stack $\mM$ considered in (\ref{Coh})
is an open substack of $\widehat{\mM}$. 

Let $\oO bj(\aA_X)$ be the 
(abstract)
substack, 
\begin{align*}
\oO bj(\aA_X) \subset \widehat{\mM}, 
\end{align*}
whose $S$-valued points consist 
of $\eE \in \widehat{\mM}(S)$ satisfying 
$\eE_s \in \aA_X$
for all $s\in S$. 
The stack $\oO  bj(\aA_X)$ decomposes as 
\begin{align*}
\oO bj(\aA_X)=\coprod_{v \in \Gamma}\oO bj^{v}(\aA_X),
\end{align*}
where $\oO bj^{v}(\aA_X)$ is the stack of 
objects $E\in \aA_X$ with $\cl(E)=v$. 
As proved in~\cite[Lemma~3.16]{Tcurve1}, 
the embedding 
\begin{align*}
\oO bj^{v}(\aA_X) \subset \widehat{\mM}, 
\end{align*}
is an open immersion if $v=(n, \beta, r) \in \Gamma$
with $r=0$ or $r=1$. 
In particular
in that case, 
 $\oO bj^{v}(\aA_X)$ is an 
Artin stack locally of finite type over $\mathbb{C}$. 
In general, $\oO bj^{v}(\aA_X)$ is at least 
a locally 
constructible subset of $\widehat{\mM}$.

Let $\omega$ and $\theta$ be as in (\ref{dot}).
We define 
\begin{align*}
\widehat{\mM}_{n, \beta}(\omega, \theta)
 \subset \oO bj^{(-n, -\beta, 1)}(\aA_X), 
\end{align*}
to be the stack which parameterizes
 $Z_{\omega, \theta}$-semistable 
objects $E\in \aA_X$ with $\cl(E)=(-n, -\beta, 1)$. 
We have the following proposition. 
\begin{prop}\label{prop:Mn}
(i) The stack $\widehat{\mM}_{n, \beta}(\omega, \theta)$ is an 
Artin stack of finite type over $\mathbb{C}$. 

(ii) If $\theta$ is sufficiently close to $1$, 
then we have 
\begin{align*}
\widehat{\mM}_{n, \beta}(\omega, \theta) \cong [P_n(X, \beta)/\mathbb{G}_m], 
\end{align*}
where $\mathbb{G}_m$ acts on $P_n(X, \beta)$ trivially. 

(iii) We have the isomorphism, 
\begin{align*}
\widehat{\mM}_{n, \beta}(\omega, \theta) \stackrel{\cong}{\to}
\widehat{\mM}_{-n, \beta}(\omega, 1-\theta), 
\end{align*}
given by 
\begin{align*}
E \mapsto \dR \hH om(E, \oO_X). 
\end{align*}

(iv) We have 
\begin{align*}
\widehat{\mM}_{n, \beta}(\omega, \theta=1/2)=\emptyset, 
\end{align*}
for $\lvert n \lvert \gg 0$. 
\end{prop}
\begin{proof}
By Lemma~\ref{lem:compare}, 
the stack $\widehat{\mM}_{n, \beta}(\omega, \theta)$
is identified with the moduli stack 
of $\mu_{k\omega+i\omega}$-limit semistable 
objects $E \in \aA_{1/2}^p$ satisfying
(\ref{chE}), where $k$ is given by (\ref{2theta}). 
The results of the proposition follow
from the corresponding results 
for $\mu_{k\omega +i\omega}$-limit stability. 
Namely (i) follows from~\cite[Proposition~3.17]{Tolim2}, 
(ii) follows from~\cite[Theorem~3.21]{Tolim2},
(iii) follows from~\cite[Lemma~2.28]{Tolim}
and (iv) follows from~\cite[Lemma~4.4]{Tolim2}.
\end{proof}

\subsection{Rank one counting invariants}
Using the moduli stack $\widehat{\mM}_{n, \beta}(\omega, \theta)$, 
we are able to construct the invariant, 
\begin{align*}
\DT_{n, \beta}(\omega, \theta) \in \mathbb{Q}, 
\end{align*}
which counts $Z_{\omega, \theta}$-semistable 
$E\in \aA_X$ with $\cl(E)=(-n, -\beta, 1)$. 
Namely, suppose that any 
$Z_{\omega, \theta}$-semistable object $E\in \aA_{X}$
with $\cl(E)=(-n, -\beta, 1)$ is $Z_{\omega, \theta}$-stable. 
 (This is true if $\omega$
and $\theta$ are chosen to be generic.)
Then 
we have 
\begin{align}\label{gerbe}
\widehat{\mM}_{n, \beta}(\omega, \theta)
\cong [\widehat{M}_{n, \beta}(\omega, \theta)/\mathbb{G}_m]
\end{align}
 for an algebraic space 
$\widehat{M}_{n, \beta}(\omega, \theta)$
of finite type over $\mathbb{C}$. 
If $\nu_{M}$ is the Behrend 
function on $\widehat{M}_{n, \beta}(\omega, \theta)$, 
then we can define 
\begin{align*}
\DT_{n, \beta}(\omega, \theta)=\int_{\widehat{M}_{n, \beta}(\omega, \theta)}
\nu_{M} d\chi. 
\end{align*}
On the other hand, suppose that there is a
 strictly $Z_{\omega, \theta}$-semistable 
object $E \in \aA_X$ satisfying 
$\cl(E)=(-n, -\beta, 1)$.
Then the stack $\widehat{\mM}_{n, \beta}(\omega, \theta)$
is not written in a way (\ref{gerbe}),
and we need to modify the definition
of $\DT_{n, \beta}(\omega, \theta)$ using the Hall type algebra
as we discussed in the previous section.   
Namely we consider 
\begin{align*}
H(\aA_X) \cneq K_0(\mathrm{St}/\oO bj(\aA_X)), 
\end{align*}
and the $\ast$-product on $H(\aA_X)$
given in a similar way to (\ref{product}), 
by replacing $\widehat{\mM}$ by $\oO bj(\aA_X)$. 
By Proposition~\ref{prop:Mn},
 we can define the elements in $\hH(\aA_X)$, 
\begin{align*}
\widehat{\delta}_{n, \beta}(\omega) &=
[\mM_{n, \beta}(\omega) \stackrel{i}{\hookrightarrow}
\oO bj(\aA_X) ], \\
\widehat{\delta}_{n, \beta}(\omega, \theta)
&=[\widehat{\mM}_{n, \beta}(\omega, \theta) \hookrightarrow 
\oO bj(\aA_X)],
\end{align*} 
where $\mM_{n, \beta}(\omega)$ is the stack introduced in (\ref{substack}),
and $i$ sends $E \in \Coh_{\le 1}(X)$ to $E[-1] \in \aA_X$. 
Its `logarithm' is defined by, 
\begin{align*}
\widehat{\epsilon}_{n, \beta}(\omega, \theta)
=\sum_{\begin{subarray}{c}l \ge 1, 1 \le e \le l, 
(n_i, \beta_i) \in \mathbb{Z} \oplus H_2(X, \mathbb{Z}), \\
n_1 + \cdots +n_l=n, 
 \beta_1 + \cdots +\beta_l=\beta \\
Z_{\omega, \theta}(-n_i, -\beta_i, 0) \in \mathbb{R}_{>0} \exp(i\pi \theta), 
i\neq e. 
\end{subarray}}
&\frac{(-1)^{l-1}}{l} 
\widehat{\delta}_{n_1, \beta_1}(\omega)\ast \cdots 
\ast \widehat{\delta}_{n_{e-1}, \beta_{e-1}}(\omega) \\
&\ast \widehat{\delta}_{n_e, \beta_e}(\omega, \theta) \ast
\widehat{\delta}_{n_{e+1}, \beta_{e+1}}(\omega)\ast \cdots \ast
\widehat{\delta}_{n_{l}, \beta_l}(\omega).  
\end{align*}
Then $\DT_{n, \beta}(\omega, \theta) \in \mathbb{Q}$ can be defined by 
\begin{align*}
\DT_{n, \beta}(\omega, \theta)=\lim_{t\to 1}(t^2 -1)
P_t(-\nu \cdot \epsilon_{n, \beta}(\omega, \theta)),
\end{align*}
where $\nu$ is defined similarly to (\ref{map:nu})
by using the Behrend function on $\oO bj(\aA_X)$. 
Also see~\cite[Definition~4.1]{Tolim2}, 
\cite[Definition~4.11]{Tcurve1}.
We define the invariant $L_{n, \beta} \in \mathbb{Q}$
as follows. 
\begin{defi}
We define $L_{n, \beta} \in \mathbb{Q}$ to be 
\begin{align*}
L_{n, \beta} \cneq \DT_{n, \beta}(\omega, \theta=1/2). 
\end{align*}
\end{defi}
As a corollary of Proposition~\ref{prop:Mn}, 
we have the following: 
\begin{cor}\label{cor:inv}
(i) If $\theta$ is sufficiently close to $1$, we have 
\begin{align*}
\DT_{n, \beta}(\omega, \theta)=P_{n, \beta}.
\end{align*}

(ii) The invariant $L_{n, \beta}$ satisfies, 
\begin{align*}
L_{n, \beta}=L_{-n, \beta}, 
\end{align*}
and they are zero for $\lvert n \rvert \gg 0$. 
\end{cor}

\subsection{Wall-crossing formula}
We define the series $\DT(\omega, \theta)$ by 
\begin{align}\label{series}
\DT(\omega, \theta) \cneq \sum_{n, \beta}\DT_{n, \beta}(\omega, \theta)
q^n t^{\beta}. 
\end{align}
Similarly to~\cite[Definition~4.11]{Tcurve1}, \cite[Section~4.3]{Tcurve2}, 
the series (\ref{series}) can be defined in a
certain topological vector space
for $0<\theta<1/2$.
Also as in~\cite[Subsection~5.1]{Tcurve1}, 
it is straightforward to check
the existence of wall and chamber structure on 
the space $\Stab_{\Gamma_{\bullet}}(\dD_X)$. 
Therefore the
following limiting series makes sense for $\phi \in (0, 1/2)$, 
\begin{align*}
\DT(\omega, \phi_{\pm}) \cneq \lim_{\theta \to \phi \pm 0}
\DT(\omega, \theta). 
\end{align*}
Using Joyce-Song's wall-crossing formula~\cite{JS}
and assuming the result by Behrend-Getzler
\footnote{The result of~\cite{BG} is not yet 
written at the moment the author writes this
manuscript}~\cite{BG},
we have the following theorem.
(Also see Remark~\ref{rmk:BG} and~\cite[Remark~2.32, Conjecture~4.3]{Tcurve2}.)
\begin{thm}
For $0<\phi<1/2$, we have the following formula, 
\begin{align}\label{WCF}
\DT(\omega, \phi_{+})=
\DT(\omega, \phi_{-})
\cdot \prod_{\begin{subarray}{c} n>0, \beta >0 \\
-n+(\omega \cdot \beta)i \in \mathbb{R}_{>0}
e^{i\pi \phi}
\end{subarray}}
\exp\left((-1)^{n-1}nN_{n, \beta}q^n t^{\beta} \right). 
\end{align}
\end{thm}
\begin{proof}
Let us fix $\omega$ and consider the 
subset 
\begin{align*}
\vV \subset \Stab_{\Gamma_{\bullet}}(\dD_X),
\end{align*}
defined by the image of the map (\ref{cmap}). 
Then it is easy to check
that the subspace $\vV$
satisfies the assumptions of~\cite[Assumption~4.1]{Tcurve1}.
Therefore the result follows from~\cite[Theorem~5.8, 
Theorem~8.10 (arXiv version)]{Tcurve1}.
\end{proof}
As a corollary of the above theorem, we obtain 
the desired product expansion (\ref{product1}). 
\begin{cor}
We have the formula, 
\begin{align}\label{PTprod}
\PT(X)&=\prod_{n>0, \beta >0}
\exp\left((-1)^{n-1} nN_{n, \beta}q^n t^{\beta}\right)
\left(\sum_{n, \beta}L_{n, \beta}q^n t^{\beta} \right). 
\end{align}
\end{cor}
\begin{proof}
By Corollary~\ref{cor:inv}, we have 
\begin{align*}
\lim_{\theta \to 1} \DT(\omega, \theta)=\PT(X). 
\end{align*}
On the other hand, 
note that if $F\in \Coh_{\le 1}(X)$ satisfies 
\begin{align*}
Z_{\omega, 1/2}(F[-1]) \in \mathbb{R}_{>0}\sqrt{-1},
\end{align*}
then $\chi(F)=0$. Using this fact and following the 
argument of~\cite[Theorem~5.8, Theorem~8.10]{Tcurve1}, 
it can be checked that
\begin{align*}
\lim_{\theta \to 1/2}\DT(\omega, \theta)&=
\DT(\omega, \theta=1/2), \\
&=\sum_{n, \beta}L_{n, \beta}q^n t^{\beta}. 
\end{align*}
Therefore applying wall-crossing formula (\ref{WCF}) 
from $\theta =1/2$ to $\theta \to 1$, 
we obtain the formula (\ref{PTprod}). 
(See~\cite[Corollary~5.11]{Tcurve1} to justify this argument.)
\end{proof}

\section{Product expansion formula}\label{sec:Prod}
In this section, we discuss a conjectural 
product expansion formula of the series $\PT(X)$, 
and see how it is related to our formula (\ref{PTprod}).  
It leads to a conjectural multi-covering formula 
of the invariant $N_{n, \beta}$, and we will 
give its evidence in a specific example. 
\subsection{Gopakumar-Vafa formula}
For $g\ge 0$ and $\beta \in H_2(X, \mathbb{Z})$, 
the GW invariant $N_{g, \beta}^{\rm{GW}} \in \mathbb{Q}$
is not an integer in general. 
However Gopakumar-Vafa~\cite{GV} claims
the following integrality of $N_{g, \beta}^{\rm{GW}}$, 
based on the string duality between Type IIA
string theory and M-theory. 
\begin{conj}
There are integers 
\begin{align*}
n_{g}^{\beta} \in \mathbb{Z}, 
\mbox{ for }
g \ge 0, \ \beta \in H_2(X, \mathbb{Z}), 
\end{align*}
such that we have 
\begin{align}\label{GWGV}
\sum_{g\ge 0, \beta>0}N_{g, \beta}^{\rm{GW}}
\lambda^{2g-2}t^{\beta}=
\sum_{g\ge 0, \beta>0, k \in \mathbb{Z}_{\ge 1}}
\frac{n_{g}^{\beta}}{k}
\left(2\sin\left(\frac{k\lambda}{2} \right)^{2g-2} \right)t^{k\beta}.
\end{align}
\end{conj}
The invariant $n_{g}^{\beta} \in \mathbb{Z}$
is called a \textit{Gopakumar-Vafa invariant}. 
The LHS of (\ref{GWGV}) can be always 
written as in the RHS of (\ref{GWGV})
for some $n_{g}^{\beta} \in \mathbb{Q}$, 
but the integrality of $n_{g}^{\beta}$ is not obvious. 
The above conjecture is implied by 
GW/DT/PT correspondence, noting that
 DT or PT invariants are integers. 
(cf.~\cite[Theorem~3.19]{PT}.)

Now let us believe GW/DT/PT correspondence and 
write GW generating series in the Gopakumar-Vafa form
(\ref{GWGV}). Then  
the series $\PT(X)$ should be written as a certain 
conjectural formula 
involving $n_{g}^{\beta}$. 
The expected formula is formulated 
in~\cite{Katz2}:
\begin{conj}\label{conj:PTGV}
There are integers 
\begin{align*}
n_{g}^{\beta} \in \mathbb{Z}, \mbox{ for }
g \ge 0, \ \beta \in H_2(X, \mathbb{Z}), 
\end{align*}
such that we have 
\begin{align}\label{GVform}
\PT(X)=\prod_{\beta >0} \left( \prod_{j=1}^{\infty}
(1-(-q)^j t^{\beta})^{j n_{0}^{\beta}}
 \cdot \prod_{g=1}^{\infty}\prod_{k=0}^{2g-2}
(1-(-q)^{g-1-k}t^{\beta})^{(-1)^{k+g}n_{g}^{\beta}
\left(\begin{subarray}{c}
2g-2 \\
k
\end{subarray} \right)} \right). 
\end{align}
\end{conj}
The above conjecture is nothing 
but the strong rationality 
conjecture discussed in~\cite{PT}. 
In what follows we discuss the relationship 
between the formulas (\ref{PTprod}) and (\ref{GVform}). 

\subsection{Multi-covering formula of $N_{n, \beta}$}
First let us take the logarithm of the RHS of (\ref{GVform}). 
Then we obtain 
\begin{align}\notag
&\log 
\prod_{\beta>0}\prod_{j=1}^{\infty}
(1-(-q)^j t^{\beta})^{jn_{0}^{\beta}}
\prod_{g=1}^{\infty}\prod_{k=0}^{2g-2}
(1-(-q)^{g-1-k}t^{\beta})^{(-1)^{k+g}n_{g}^{\beta}
\left(\begin{subarray}{c}
2g-2 \\
k
\end{subarray} \right)} \\
&=
\sum_{\beta>0} 
\sum_{j=1}^{\infty}
jn_{0}^{\beta}\log \left(1-(-q)^j t^{\beta}  \right)
 \\
\notag
&\qquad +\sum_{\beta>0}\sum_{g=1}^{\infty}\sum_{k=0}^{2g-2}
(-1)^{k+g}n_{g}^{\beta}
\binom{2g-2}{k}
\log\left(1-(-q)^{g-1-k}t^{\beta} \right) \\
\notag
&=\sum_{\beta>0} 
\sum_{j=1}^{\infty}
jn_{0}^{\beta}\sum_{k\ge 1}\frac{(-1)^{jk-1}q^{jk}}{k}
t^{k\beta} \\
\label{2nd}
&\qquad +\sum_{\beta>0}\sum_{g=1}^{\infty}\sum_{a \ge 1}
\frac{n_g^{\beta}}{a}\sum_{k=0}^{2g-2}
\binom{2g-2}{k}
\left\{-(-q)^a \right\}^{g-1-k}
t^{a\beta}.
\end{align}
The first term of (\ref{2nd})
is written as 
\begin{align}
\sum_{\beta>0} 
\sum_{n=1}^{\infty}
\sum_{k\ge 1, k|(\beta, n)}
\frac{(-1)^{n-1}n}{k^2}
n_{0}^{\beta/k}q^n t^{\beta},
\end{align}
and the coefficient of $t^{\beta}$ 
is an element of $q\mathbb{Q}\db[q\db]$. 
As for the second term of (\ref{2nd}), we set 
\begin{align}\notag
f_{g}(q) &\cneq \sum_{k=0}^{2g-2} \binom{2g-2}{k} q^{g-1-k} \\
\label{def:fg}
&=q^{1-g}(1+q)^{2g-2}.
\end{align}
Then the second term of (\ref{2nd}) is written as 
\begin{align}
\sum_{\beta>0}\label{2ndterm}
\sum_{g=1}^{\infty}
\sum_{\begin{subarray}{c}a \ge 1, 
a|\beta
\end{subarray}}
\frac{n_g^{\beta/a}}{a}f_g(-(-q)^a)t^{\beta}. 
\end{align}
Note that the coefficient of $t^{\beta}$
in (\ref{2ndterm}) is a polynomial of $q^{\pm 1}$
invariant under $q\leftrightarrow 1/q$. 

Next taking the logarithm of (\ref{PTprod}), we obtain 
\begin{align}\label{logPT}
\log \PT(X)
=\sum_{\beta>0}\sum_{n>0}(-1)^{n-1}nN_{n, \beta}q^n t^{\beta}
+\log\left( \sum_{n, \beta}L_{n, \beta}q^n t^{\beta} \right).
\end{align}
The coefficient of $t^{\beta}$ in
the first term of the RHS of (\ref{logPT})
is an element of
 $q\mathbb{Q}\db[q\db]$. 
We set 
\begin{align}\label{logL}
\sum_{\beta>0}L_{\beta}(q)t^{\beta}
\cneq \log\left( \sum_{n, \beta}L_{n, \beta}q^n t^{\beta} \right).
\end{align}
Then 
$L_{\beta}(q)$ is a polynomial of $q^{\pm 1}$
which is 
invariant under $q\leftrightarrow 1/q$. 

For a Laurent series $F(q)$ in $q$, note that the 
decomposition 
\begin{align*}
&F(q)=F_1(q)+F_2(q), \\
& F_1(q) \in q \mathbb{Q}\db[ q \db], 
F_2(q) \in \mathbb{C}[q^{\pm}1],
\end{align*}
is unique if $F_2(q)$ is invariant under $q\leftrightarrow 1/q$. 
Hence 
if Conjecture~\ref{conj:PTGV} holds, 
the comparison of (\ref{2nd}) with (\ref{logPT})
gives 
\begin{align}\label{comp1}
\sum_{n>0}(-1)^{n-1}nN_{n, \beta}q^n
&=
\sum_{n=1}^{\infty}
\sum_{k\ge 1, k|(\beta, n)}
\frac{(-1)^{n-1}n}{k^2}
n_{0}^{\beta/k}q^n, \\
\label{comp2}
L_{\beta}(q) &=
\sum_{g=1}^{\infty}
\sum_{\begin{subarray}{c}a \ge 1, 
a|\beta
\end{subarray}}
\frac{n_g^{\beta/a}}{a}f_g(-(-q)^a).
\end{align}
By looking at the coefficient of $q$ in (\ref{comp1}), 
we obtain 
\begin{align*}
N_{1, \beta}=n_{0, \beta}. 
\end{align*}
Then by looking at the coefficient of $q^n$, 
we obtain the following conjectural formula. 
\begin{conj}\label{conj:N}
We have the following formula, 
\begin{align}\label{multi}
N_{n, \beta}=\sum_{k\ge 1, k|(n, \beta)}
\frac{1}{k^2}N_{1, \beta/k}. 
\end{align}
\end{conj}
By the above argument, if Conjecture~\ref{conj:N} is true, 
then $n_{0}^{\beta} =N_{1, \beta}$ satisfies 
the equation (\ref{comp1}).
Note that $N_{1, \beta}$ is an integer
since the vector $(1, \beta)$ is primitive. 

Also the equation (\ref{comp2})
gives a way to write down $n_{g}^{\beta}$ for $g\ge 1$
in terms of $L_{n, \beta}$. 
Namely if $G(q) \in \mathbb{Q}[q^{\pm 1}]$ is invariant 
under $q\leftrightarrow 1/q$, then there is a unique 
way to write $G(q)$ as
\begin{align*}
G(q)=\sum_{g=1}^{N}a_g f_g(q),
\end{align*}
with $a_g \in \mathbb{Q}$. 
Hence
we are able to write down $n_{g}^{\beta}$ in
terms $L_{n, \beta}$ using the equation (\ref{comp2}) 
recursively. 
For instance, 
as we will see in Theorem~\ref{thm:higher},
 we have 
\begin{align}\label{n1}
n_{1}^{\beta}=\sum_{n}(-1)^n L_{n, \beta}
-\frac{1}{2}\sum_{n_1, n_2}\sum_{\beta_1 +\beta_2=\beta}
(-1)^{n_1+n_2}L_{n_1, \beta_1}L_{n_2, \beta_2}
+ \cdots,
\end{align}
if $\beta$ is a primitive curve class. 
The integrality of $n_{g}^{\beta}$ for $g\ge 1$
is not obvious from the expression
of $n_{g}^{\beta}$ in terms of $L_{n, \beta}$, as in (\ref{n1}).
However by \cite[Theorem~3.19]{PT}, if $\PT(X)$ is once written as 
a product expansion (\ref{GVform}), 
then the integrality of $n_{g}^{\beta}$ 
follows from the integrality of $P_{n, \beta} \in \mathbb{Z}$. 
As a summary, we obtain the following. 
\begin{thm}\label{thm:PTGV}
Conjecture~\ref{conj:PTGV} is equivalent to Conjecture~\ref{conj:N}. 
In that case, we have 
\begin{align*}
n_{0}^{\beta}=N_{1, \beta}, 
\end{align*}
and there is a way to write 
down $n_{g}^{\beta}$ for $g\ge 1$ in terms of $L_{n, \beta}$. 
\end{thm} 
\begin{rmk}
The invariant $N_{1, \beta}$ is nothing 
but Katz's definition of genus zero Gopakumar-Vafa
invariant~\cite{Katz}. 
\end{rmk}
\subsection{Higher genus Gopakumar-Vafa invariants}
As we observed in Theorem~\ref{thm:PTGV}, 
if we assume Conjecture~\ref{conj:PTGV},
then $n_{g}^{\beta}$ is written in 
terms of $L_{n, \beta}$. 
The purpose of this subsection is to give
its explicit formula. 

For $m\ge 0$, we set $h_m(q)$ by 
\begin{align*}
h_m(q)=\left\{
\begin{array}{cc}
1, & m=0, \\
q^m +q^{-m}, & m\ge 1.
\end{array}   \right. 
\end{align*} 
Let $f_{g}(q)$ be the function defined 
by (\ref{def:fg}).
Then for $g\ge 1$, we have 
\begin{align}\label{f=h}
f_{g}(q)=\sum_{m=0}^{g-1}\binom{2g-2}{g-1+m}
h_m(q).
\end{align}
There is an inversion formula of (\ref{f=h}). 
Namely there are
$c_{g}^{(m)} \in \mathbb{Z}$ such that 
\begin{align}\label{h=c}
h_m(q)=\sum_{g=1}^{m+1} c_g^{(m)} f_g(q). 
\end{align}
An elementary calculation shows that 
$c_g^{(m)}$ is given by 
\begin{align}\label{elemental}
c_g^{(m)}=(-1)^{m+g-1}
\left\{ \binom{m+g}{2g-1} -\binom{m+g-2}{2g-1} \right\}.
\end{align}
The M$\ddot{\rm{o}}$bius function 
on $\mathbb{Z}_{\ge 1}$
is defined as follows:
\begin{align*}
\mu(n)=\left\{ 
\begin{array}{cc}
(-1)^{\omega(n)}, & \mbox{ if } n \mbox{ is square free }, \\
0, & \mbox{ otherwise. }
\end{array}
\right.
\end{align*}
Here $\omega(n)$ is the number of 
distinct prime factors of $n$. 
Then by (\ref{comp2}) and 
the M$\ddot{\rm{o}}$bius
inversion formula, we have 
\begin{align}\label{Minv}
\sum_{g \ge 1}n_g^{\beta}
f_g(q) =\sum_{a\ge 1, a|\beta}\frac{\mu(a)}{a}
L_{\beta/a}(-(-q)^a).
\end{align}
If we write 
\begin{align}\label{Lprime}
L_{\beta}(q)=\sum_{n, \beta}L_{n, \beta}'q^n,
\end{align}
for $L_{n, \beta}' \in \mathbb{Q}$, then 
we have 
\begin{align*}
(\ref{Minv}) &=
\sum_{a\ge 1, a|\beta} \frac{\mu(a)}{a}
\sum_{n \in \mathbb{Z}} L_{n, \beta/a}' (-1)^{na+n}q^{na} \\
&=\sum_{a\ge 1, a|\beta} \frac{\mu(a)}{a}
\sum_{n\ge 0}(-1)^{na+n}L_{n, \beta/a}' h_{na} \\
&=\sum_{n \ge 0}
\sum_{a\ge 1, a|(n, \beta)}
\frac{\mu(a)}{a}(-1)^{n+n/a}L_{n/a, \beta/a}' h_n \\
&= \sum_{g \ge 1} \left(
\sum_{n\ge g-1} \sum_{a\ge 1, a|(n, \beta)}
\frac{\mu(a)}{a}(-1)^{n+n/a}L_{n/a, \beta/a}' c_g^{(n)} \right) f_g(q).
\end{align*}
Here we have used (\ref{h=c}) for the last equality. 
On the other hand, comparing (\ref{logL}) with (\ref{Lprime}), 
we have 
\begin{align*}
L_{n, \beta}'=\sum_{l \ge 1}\frac{(-1)^{l-1}}{l}
\sum_{\begin{subarray}{c}
n_1 + \cdots +n_l =n, \\
\beta_1 + \cdots +\beta_l=\beta
\end{subarray}}
\prod_{i=1}^{l} L_{n_i, \beta_i}. 
\end{align*}
Also using the formula (\ref{elemental}) for $c_g^{(n)}$, 
we obtain the following result. 
\begin{thm}\label{thm:higher}
Suppose that Conjecture~\ref{conj:PTGV} is 
true. Then $n_{0}^{\beta}=N_{1, \beta}$
and $n_{g}^{\beta}$ for $g\ge 1$ is given by 
\begin{align*}
n_{g}^{\beta}=
\sum_{\begin{subarray}{c}
n\ge g-1, \\
a\ge 1, a|(n, \beta)
\end{subarray}}
\sum_{\begin{subarray}{c}
l\ge 1, \\
n_1 + \cdots +n_l =n/a, \\
\beta_1 + \cdots +\beta_l=\beta/a
\end{subarray}}
\frac{\mu(a)}{al}(-1)^{l+g+n/a}
\left\{\binom{n+g}{2g-1}-\binom{n+g-2}{2g-1}  \right\}. 
\prod_{i=1}^{l}L_{n_i, \beta_i}.
\end{align*}
\end{thm}

\subsection{Example: Weierstrass model}
We will prove Conjecture~\ref{conj:PTGV} 
and compute $n_{g, \beta}$
in the following specific example.  
Let $S$ be a smooth projective del-Pezzo 
surface over $\mathbb{C}$. 
Take general elements, 
\begin{align*}
f \in \Gamma(S, \oO_S(-4K_S)), \quad 
g \in \Gamma(S, \oO_S(-6K_S)). 
\end{align*}
We construct a Calabi-Yau 3-fold with 
an elliptic fibration, 
\begin{align*}
\pi \colon X \to S, 
\end{align*}
by the defining equation 
\begin{align*}
y^2 =x^3 +fx +g,
\end{align*}
in the projective bundle, 
\begin{align*}
\mathcal{P}roj \mathrm{Sym}_{S}^{\bullet}
(\oO_{S} \oplus \oO_S(-2K_S) \oplus \oO_S(-3K_S)) \to S. 
\end{align*}
Here $x$ and $y$ are local sections of 
$\oO_S(-2K_S)$ and $\oO_S(-3K_S)$ respectively. 
A Calabi-Yau 3-fold $X$ constructed in this way is called 
a \textit{Weierstrass model}. 
A general fiber of $\pi \colon X \to S$ is 
a smooth elliptic curve, and any 
singular fiber is either a nodal or 
cuspidal plane curve. 

Let $F \subset X$ be a general fiber of $\pi$. 
We study the following series, 
\begin{align*}
\PT(X/S) \cneq \sum_{n, m}
\PT_{n, m[F]}q^n t^{m}. 
\end{align*}
By the formula (\ref{PTprod}), we have the 
product expansion formula, 
\begin{align}\label{PTXS}
\PT(X/S) =\prod_{n>0, m>0}
\exp\left( (-1)^{n-1}n N_{n, m[F]} q^n t^{m}\right) \left( \sum_{n, m}
L_{n, m[F]}q^n t^{m} \right).
\end{align}
In what follows, we omit $[F]$ in the notation for simplicity.
So for instance, we write $N_{n, m[F]}$ as $N_{n, m}$. 
\begin{prop}\label{prop:Nmult}
The invariant $N_{n, m}$ satisfies the formula (\ref{multi}), and 
\begin{align*}
N_{1, m}=-\chi(X). 
\end{align*}
\end{prop} 
\begin{proof}
Let $\omega_X$ be an ample divisor on $X$. 
Let 
\begin{align*}
\mM_{n, m}^{s}(\omega_X) \subset \mM_{n, m}(\omega_X),
\end{align*}
be the substack corresponding to 
$Z_{\omega_X}$-stable 
objects in $\Coh_{\le 1}(X)$, 
introduced in Example~\ref{exam:stab} (iii). 
Note that if $E \in \Coh_{\le 1}(X)$ represents a closed point 
of $\mM_{n, m}^{s}(\omega_X)$, then 
$E$ is 
written as 
\begin{align}\label{include}
E\cong i_{p \ast} E',
\end{align}
for some stable sheaf $E'$
on an elliptic fiber
$\pi^{-1}(p)$ for some 
$p \in S$. 
Here $i_{p} \colon \pi^{-1}(p) \hookrightarrow 
X$ is the inclusion. 
By the classification of stable sheaves on 
the fibers of $\pi$ given in~\cite{BBDG}, we have
\begin{align}\label{class}
\mM_{n, m}^{s}(\omega_X)= \emptyset, \quad 
\mbox{ if } \mathrm{g.c.d.}(n, m)>1. 
\end{align}
Assume that $\mathrm{g.c.d.}(n, m)=1$. 
Let 
\begin{align*}
Y \to S, 
\end{align*}
be the relative moduli space of $Z_{\omega_X}$-stable 
sheaves $E$ on the fibers of $\pi \colon X \to S$, satisfying 
\begin{align}\label{Echi}
[E]=m[F], \quad \chi(E)=n. 
\end{align}
By
the condition $\mathrm{g.c.d.}(n, m)=1$
and the result of~\cite{B-M2},
the variety $Y$ is smooth projective, irreducible, and 
there is a derived equivalence, 
\begin{align}\label{Dequiv}
\Phi \colon D^b \Coh(X) \stackrel{\sim}{\to}
D^b \Coh(Y), 
\end{align}
which takes any $Z_{\omega_X}$-stable sheaf
satisfying (\ref{Echi}) to an object of 
the form $\oO_y$ for a closed point $y\in Y$.  
For $d \in \mathbb{Z}_{\ge 1}$, take a $\mathbb{C}$-valued point, 
\begin{align*}
[E] \in \mM_{(dn, dm)}(\omega_X). 
\end{align*}
By (\ref{class}),
any Jordan-H$\ddot{\rm{o}}$lder factor of $E$ 
determines a closed point in $\mM_{n, m}(\omega_X)$.
Hence the equivalence $\Phi$ induces the isomorphism, 
\begin{align*}
\mM_{(dn, dm)}(\omega_X) \stackrel{\sim}{\to}
\mM_{(d, 0)}(\omega_Y). 
\end{align*} 
 Here $\omega_Y$ is an arbitrary polarization on $Y$. 
(Obviously the RHS does not depend on $\omega_Y$.)
Therefore we obtain that 
\begin{align*}
N_{dn, dm}(\omega_X) &= N_{d, 0}(\omega_Y) \\
&= -\chi(Y)\sum_{k\ge 1, k|d}\frac{1}{k^2}, \\
&=-\chi(X) \sum_{k\ge 1, k|d}\frac{1}{k^2}.
\end{align*}
Here the second equality follows from (\ref{D0}) 
and the last equality follows from the derived 
equivalence (\ref{Dequiv}).
Therefore we obtain the desired result. 
\end{proof}
Next we compute the invariants 
$L_{n, m}$. 
\begin{prop}\label{prop:L}
We have 
$L_{n, m}=0$ for $n\neq 0$, and 
\begin{align*}
L_{0, m}=\chi(\Hilb_m(S)). 
\end{align*}
Here $\Hilb_m(S)$ is the Hilbert scheme of 
$m$-points in $S$. 
\end{prop}
\begin{proof}
Let us take 
an ample divisor $\omega$ on $X$ and a 
 stable pair 
\begin{align}\label{spair}
s\colon \oO_X \to E, 
\end{align}
with $E$ supported on fibers of $\pi$. 
By taking the Harder-Narasimhan 
filtration and Jordan-H$\ddot{\rm{o}}$lder
filtration
with respect to $Z_{\omega}$-stability, 
(cf.~Example~\ref{exam:stab} (iii),)
 we can take a 
filtration of $E$, 
\begin{align*}
0=E_0 \subset E_1 \subset E_2 \subset \cdots \subset E_N=E, 
\end{align*}
such that each $F_i=E_i/E_{i-1}$ is $Z_{\omega}$-stable with 
\begin{align}\label{ineqF}
\arg Z_{\omega}(F_i) \ge \arg Z_{\omega}(F_{i+1}), 
\end{align}
for all $i$. 
Note that each $F_j$ is written as $i_{p\ast}F_j'$
for a stable sheaf $F_j'$ on $\pi^{-1}(p)$
as in (\ref{include}). 
Also the composition, 
\begin{align*}
\oO_X \stackrel{s}{\to} E \to E/E_{N-1}=F_N, 
\end{align*}
should be non-zero since $s$ is surjective in dimension one. 
Therefore 
\begin{align*}
\Hom_X(\oO_X, F_N) \cong \Hom_{X_p}(\oO_{X_p}, F_N) \neq 0,
\end{align*} 
which implies
that 
\begin{align*}
\arg Z_{\omega}(F_N) \ge \arg Z_{\omega}(\oO_{X_p})=\pi/2.
\end{align*}
Combined with
 the inequality (\ref{ineqF}), we conclude that $\chi(E) \ge 0$. 

The above argument shows that $P_n(X, m)$ is empty 
for $n<0$, hence $P_{n, m}=0$ for $n<0$. 
By the formula (\ref{PTXS})
and the symmetry $L_{n, m}=L_{-n, m}$, we conclude 
\begin{align*}
L_{n, m}=0, \quad \mbox{ if } n\neq 0. 
\end{align*} 
Let us compute $L_{0, m}$. By substituting $q=0$ 
in the formula (\ref{PTXS}), we have 
\begin{align}\label{L=P}
L_{0, m}=P_{0, m}. 
\end{align}
Suppose that a stable pair (\ref{spair}) 
satisfies $\chi(E)=0$. 
Then the above argument shows that 
$F_N \cong \oO_{X_p}$, and we obtain a morphism 
\begin{align*}
I_{p} \to E_{N-1}, 
\end{align*}
which is surjective in dimension one. 
Here $I_{p}$ is the ideal sheaf of 
$\pi^{-1}(p)$.  
Repeating the above argument, we 
see that 
\begin{align}\label{pair:cond}
F_i \cong \oO_{X_p}, \ \Cok(s)=0,
\end{align}  
for all $i$. 
It is easy to see that a pair (\ref{spair})
satisfying the property (\ref{pair:cond})
is obtained by the pull-back, 
\begin{align*}
\oO_S \twoheadrightarrow \oO_W, 
\end{align*}
for a zero dimensional subscheme
$W \subset S$ of length $m$. 
Therefore we have the isomorphism, 
\begin{align*}
P_{0}(X, m) \cong \Hilb_m(S),
\end{align*} 
and 
\begin{align*}
P_{0, m} =\chi(\Hilb_m(S)). 
\end{align*}
Combined with (\ref{L=P}), we obtain the 
desired result. 
\end{proof}
Combining the above two proposition, 
we obtain the following theorem. 
\begin{thm}
We have the following formula, 
\begin{align}\label{GVform2}
\PT(X/S)=\prod_{m\ge 1, j\ge 1}
(1-(-q)^j t^{m})^{-j\chi(X)}
(1-t^{m})^{-\chi(S)}. 
\end{align}
\end{thm}
\begin{proof}
By Proposition~\ref{prop:Nmult}
and Theorem~\ref{thm:PTGV},
the series $\PT(X/S)$ is 
written as a Gopakumar-Vafa form (\ref{GVform})
with $n_{0}^{m}$ equasl to $-\chi(X)$ 
for all $m \ge 1$. Also Proposition~\ref{prop:L}
implies that 
\begin{align*}
\sum_{n, m}L_{n, m}q^n t^{m} &=
\sum_{m}L_{0, m}t^{m} \\
&= \sum_{m}\chi(\Hilb_m(S))t^m \\
&=\prod_{m\ge 1}(1-t^m)^{-\chi(S)}. 
\end{align*}
Here the last equality is G$\ddot{\rm{o}}$ttsche's formula~\cite{Got}. 
Therefore we have the desired formula.  
\end{proof}

Institute for the Physics and 
Mathematics of the Universe, University of Tokyo

\textit{E-mail address}: yukinobu.toda@ipmu.jp

\end{document}